\documentclass[a4paper,12pt]{article}
\input xy
\xyoption{all}

\usepackage{amssymb,amscd}
\usepackage{amsthm}
\usepackage{amsxtra}
\usepackage{mathrsfs}
\usepackage{array,float}
\usepackage{amsfonts}
\usepackage{xr}
\usepackage{pb-diagram}
\usepackage{color}

\begin{document}

\def\sect{\section}

\newtheorem{thm}{Theorem}[section]
\newtheorem{cor}[thm]{Corollary}
\newtheorem{lem}[thm]{Lemma}
\newtheorem{prop}[thm]{Proposition}
\newtheorem{propconstr}[thm]{Proposition-Construction}
\newtheorem{pro}[thm]{Proposition}

\theoremstyle{definition}
\newtheorem{para}[thm]{}
\newtheorem{ax}[thm]{Axiom}
\newtheorem{conj}[thm]{Conjecture}
\newtheorem{defn}[thm]{Definition}
\newtheorem{notation}[thm]{Notation}
\newtheorem{rem}[thm]{Remark}
\newtheorem{remark}[thm]{Remark}
\newtheorem{question}[thm]{Question}
\newtheorem{example}[thm]{Example}
\newtheorem{problem}[thm]{Problem}
\newtheorem{excercise}[thm]{Exercise}
\newtheorem{ex}[thm]{Exercise}

\newcommand{\Z}{\mathbb{Z}}
\newcommand{\Q}{\mathbb{Q}}
\newcommand{\N}{\mathbb{N}}

\def\mL{{\mathcal L}}
\def\C{{\mathcal C}}
\def\cF{{\mathcal F}}

\overfullrule=0pt

\def\si{\sigma}
\def\prf{\smallskip\noindent{\it         Proof}. }
\def\call{{\mathcal L}}
\def\nat{{\mathbb  N}}
\def\la{\langle}
\def\ra{\rangle}
\def\inv{^{-1}}
\def\ld{{\rm    ld}}
\def\trdeg{{tr.deg}}
\def\dim{{\rm   dim}}
\def\th{{\rm    Th}}
\def\rest{{\lower       .25     em      \hbox{$\vert$}}}
\def\ch{{\rm    char}}
\def\zee{{\mathbb  Z}}
\def\conc{^\frown}
\def\acl{acl_\si}
\def\cls{cl_\si}
\def\cals{{\cal S}}
\def\mult{{\rm  Mult}}
\def\calv{{\mathcal V}}
\def\aut{{\rm   Aut}}
\def\ffi{{\mathbb  F}}
\def\F{{\mathbb  F}}
\def\ffiti{\tilde{\mathbb          F}}
\def\degs{deg_\si}
\def\calx{{\mathcal X}}
\def\gal{{\mathcal G}al}
\def\cl{{\rm cl}}
\def\loc{{\rm locus}}
\def\calg{{\mathcal G}}
\def\H{{\mathcal H}}
\def\calq{{\mathcal Q}}
\def\calr{{\mathcal R}}
\def\caly{{\mathcal Y}}
\def\aff{{\mathbb A}}
\def\cali{{\cal I}}
\def\calu{{\cal U}}
\def\epsilon{\varepsilon} 
\def\U{{\mathcal U}}
\def\V{{\mathcal V}}
\def\rat{{\mathbb Q}}
\def\ga{{\mathbb G}_a}
\def\gm{{\mathbb G}_m}
\def\cee{{\mathbb C}}
\def\ree{{\mathbb R}}
\def\frob{{\rm Frob}}
\def\Frob{{\rm Frob}}
\def\fix{{\rm Fix}}
\def\Uu{{\mathcal U}}
\def\Kk{{\mathcal K}}
\def\proj{{\mathbb P}}
\def\sym{{\rm Sym}}
 
\def\dcl{{\rm dcl}}
\def\calm{{\mathcal M}}

\font\helpp=cmsy5
\def\semdp
{\hbox{$\times\kern-.23em\lower-.1em\hbox{\helpp\char'152}$}\,}

\def\dnfo{\,\raise.2em\hbox{$\,\mathrel|\kern-.9em\lower.35em\hbox{$\smile$}
$}}
\def\dnf#1{\lower1em\hbox{$\buildrel\dnfo\over{\scriptstyle #1}$}}
\def\dfo{\;\raise.2em\hbox{$\mathrel|\kern-.9em\lower.35em\hbox{$\smile$}
\kern-.7em\hbox{\char'57}$}\;}
\def\df#1{\lower1em\hbox{$\buildrel\dfo\over{\scriptstyle #1}$}}        
\def\stab{{\rm Stab}}
\def\qfcb{\hbox{qf-Cb}}
\def\perf{^{\rm perf}}
\def\sipm{\si^{\pm 1}}
\newcommand{\fgc}{\ffi_p[[G/C]]}
\newcommand{\fgone}{\ffi_p[[G/G_1]]}
\newcommand{\fgtwo}{\ffi_p[[G/G_2]]}
\newcommand{\fhc}{\ffi_p[[H/C]]}
\newcommand{\fhone}{\ffi_p[[H/H_1]]}
\newcommand{\fhtwo}{\ffi_p[[H/H_2]]}
\newcommand{\kfgc}{\ffi_p[[K\backslash G/C]]}
\newcommand{\kfgone}{\ffi_p[[K\backslash G/G_1]]}
\newcommand{\kfgtwo}{\ffi_p[[K\backslash G/G_2]]}
\newcommand{\kfhc}{\ffi_p[[K\cap H\backslash H/C]]}
\newcommand{\kfhone}{\ffi_p[[K\cap H\backslash H/H_1]]}
\newcommand{\kfhtwo}{\ffi_p[[K\cap H\backslash H/H_2]]}
\newcommand{\PZed}[1]{\textcolor{blue}{#1}}
\newcommand{\ZCed}[1]{\textcolor{green}{#1}}
\newcommand{\red}[1]{\textcolor{red}{#1}}
\newcommand{\G} {\mathcal G}
\newcommand{\cupdot}{\mathbin{\mathaccent\cdot\cup}}
\newcommand{\bigcupdot}{\hspace{6pt}\cdot \hspace{-10pt}\bigcup}
\def\vlabel{\label}

\title{Pro-p groups acting on trees with finitely many maximal vertex stabilizers up to conjugation }

 \author{Zo\'e Chatzidakis and Pavel Zalesskii}

\date{}
\maketitle

\begin{abstract}  We prove that a finitely generated pro-$p$ group $G$ acting on a pro-$p$ tree $T$ splits as a free amalgamated pro-$p$ product or a pro-$p$ HNN-extension over an edge stabilizer. If $G$ acts with finitely many vertex stabilizers up to conjugation we show that it is the fundamental pro-$p$ group of a finite graph of pro-$p$ groups $(\G,\Gamma)$ with edge and vertex groups being stabilizers of certain vertices and edges of $T$ respectively. If edge stabilizers are procyclic, we give a bound on $\Gamma$ in terms of the minimal number of generators of $G$. We also give a criterion for a pro-$p$ group $G$ to be accessible in terms of the first cohomology $H^1(G,\F_p[[G]])$.

\end{abstract}

\section{Introduction}

The dramatic advance  of  classical combinatorial group theory happened
in the 1970's, when the  Bass-Serre theory of groups acting on trees
changed completely the face of the theory. 

The profinite version of Bass-Serre  theory was developed by Luis Ribes,
Oleg Melnikov and the second author because of the absence of the classical
methods of combinatorial group theory for profinite groups. However it
does not work in full strength even in the pro-$p$ case. The reason is
that if a pro-$p$ group $G$ acts on a pro-$p$ tree $T$ then  a maximal subtree of the quotient graph $G\backslash T$ does not always exist and even if it exists it does not always lift to $T$. As a consequence the pro-$p$ version of  Bass-Serre theory
does not give subgroup structure theorems the way it does in the
classical Bass-Serre theory. In fact,  for infinitely generated pro-$p$ subgroups there are counter examples.

The objective of this paper is to study the situation when $G$ has only
finitely many vertex stabilizers up to conjugation and in this case we
can prove the main Bass-Serre theory structure theorem.

\bigskip\noindent
{\bf Theorem \ref{General}.} {\em  
Let $G$ be a finitely generated  pro-$p$ group acting
on a pro-$p$ tree $T$ with finitely many maximal vertex stabilisers up to conjugation.
 Then  $G$ is  the fundamental group of a reduced finite graph
of finitely generated pro-$p$ groups $(\G,\Gamma)$, where  each vertex
group $\G(v)$ and each edge group $\G(e)$ is  a maximal vertex stabilizer $G_{\tilde v}$ and  an edge stabilizer
$G_{\tilde e}$ respectively (for some $\tilde v,\tilde e\in T$). }

\medskip
In the abstract situation a finitely generated (abstract) group $G$ acting on a tree has a  $G$-invariant subtree $D$ such that $G\backslash D$ is finite and so 
has automatically finitely many maximal vertex stabilizers up to
conjugation. In the pro-$p$ situation such an invariant subtree does not exist in general and the existence of it in the case of only finitely many stabilizers up to conjugation is not clear even if vertex stabilizers are finite. Nevertheless, for a finitely generated pro-$p$ group acting on a pro-$p$ tree we can prove a splitting theorem into an amalgamated product or an HNN-extension. 

\bigskip\noindent
{\bf Theorem \ref{splitting}.}\ {\em  Let $G$ be a finitely generated
  pro-$p$ group acting  on a pro-$p$ tree $T$ without   global fixed points. Then
  $G$ splits non-trivially as a free amalgamated pro-$p$ product or
  pro-$p$ HNN-extension over some stabiliser of an edge of $T$.}

\medskip

This in turn allows us to prove that  a non virtually cyclic pro-$p$ group acting on a pro-$p$ tree with finite edge stabilizers has more than one end.

\bigskip\noindent
{\bf Theorem \ref{number of ends}.} {\em Let $G$ be a finitely generated pro-$p$ group acting  on a pro-$p$ tree with finite edge stabilizers and without  global fixed points.   Then either $G$ is virtually cyclic and $H^1(G, \ffi_p[[G]])\cong \ffi_p$ (i.e. $G$ has two ends) or  $H^1(G, \ffi_p[[G]])$ is infinite  (i.e. $G$ has infinitely many ends).}

\medskip
Theorem \ref{splitting} raises naturally the question of accessibility;
namely whether we can continue to split  $G$ into an amalgamated free
product or HNN-extension  forever, or do we reach the situation after finitely many steps where we can not split it anymore.  The importance of this is underlined also by the following observation: if  a pro-$p$ group $G$ acting on a pro-$p$ tree $T$ is accessible with respect to splitting over edge stabilizers, then by Theorem \ref{splitting} this implies  finiteness of the maximal vertex stabilizers up to conjugation and so Theorem \ref{General} provides the structure theorem for $G$.

In the abstract situation accessibility was studied by Dunwoody
\cite{D85}, \cite{Dun93} for splitting over finite groups and in
\cite{BF} over an arbitrary family of groups. In the pro-$p$ case
accessibility was studied by Wilkes \cite{wilkes} where a finitely generated not accessible pro-$p$ group was constructed. For a finitely generated pro-$p$ group acting faithfully and irreducibly on a pro-$p$ tree (see Section \ref{preliminaries} for definitions) no such example is known.

The next theorem gives  a sufficient condition of accessibility for  a pro-$p$ group; we do not know whether the converse also holds (it holds in the abstract case). 

\bigskip\noindent
{\bf Theorem \ref{accessible-2}.}\ {\em  Let $G$ be a finitely generated pro-$p$ group. If $H^1(G,
  \F_p[[G]])$ is  a finitely generated $\F_p[[G]]$-module, then $G$ is
  accessible.}\footnote{proved by G. Wilkes independently in
  \cite{G.wilkes}.}

\medskip

We show here that finitely generated pro-$p$ groups are accessible with respect to cyclic subgroups and in fact give precise bounds.

\bigskip\noindent
{\bf Theorem \ref{cyclic General}.} \ {\em 
Let $G$ be a finitely generated pro-$p$ group acting on a pro-$p$ tree $T$ with procyclic edge stabilizers.
 Then  $G$ is  the fundamental group of a finite graph
of finitely generated pro-$p$ groups $(\G,\Gamma)$, where  each vertex
group $\G(v)$ and each edge group $\G(e)$ is conjugate into a subgroup
of a vertex stabilizer $G_{\tilde v}$ and an edge stabilizer
$G_{\tilde e}$ respectively. Moreover, $|V(\Gamma)|\leq 2d-1$, and $|E(\Gamma)|\leq 3d-2$, where $d$ is the minimal number of generators of $G$.}

\medskip
Observe that Theorem \ref{cyclic General} contrasts with  the abstract groups situation where for finitely generated groups  the result does not hold (see \cite{DJ}).

\medskip
As a corollary we deduce the bound for pro-$p$ limit groups (pro-$p$
analogs of limit groups introduced in \cite{Kochloukova2}, see Section
\ref{Generalized Accessible}  for a precise definition).

\bigskip\noindent
{\bf Corollary \ref{cor-KS}.}\ 
{\em Let $G$ be a pro-$p$ limit group. Then  $G$ is  the fundamental group of a finite graph
of finitely generated pro-$p$ groups $(\G,\Gamma)$, where   each edge group $\G(e)$ is infinite procyclic. Moreover, $|V(\Gamma)|\leq 2d-1$, and $|E(\Gamma)|\leq 3d-2$, where $d$ is the minimal number of generators of $G$.}

 \medskip
It is worth to mention that for abstract limit groups the best known estimate for $|V(\Gamma)|$ is $1+4(d(G)-1)$, proved by Richard Weidmann in \cite[Theorem 1]{W}.

\medskip
In Section \ref{Howson property} we investigate Howson's property for free products with procyclic amalgamation and HNN-extensions with procyclic associated subgroups. In Section \ref{normalizers} we apply the results of Section  \ref{Howson property} to normalizers of procyclic subgroups.

\bigskip\noindent
{\bf Theorem \ref{normalizer-3}.}\ {\em Let $C$ be a procyclic pro-$p$ group and $G=G_1\amalg_CG_2$ be a free amalgamated pro-$p$ product or a pro-$p$ HNN-extension $G={\rm HNN}(G_1,C,t)$ of Howson groups. Let $U$ be a procyclic
     subgroup of $G$   and $N=N_G(U)$.  Assume that $N_{G_i}(U^g)$ is finitely
   generated whenever $U^g\leq G_i$. If $K\leq
   G$ is finitely generated, then so is $K\cap N$.}

 \medskip

Section \ref{preliminaries} contains basic notions and  facts of the
theory of pro-$p$ groups acting on trees used in the paper. The
following sections are devoted to the proofs of the results mentioned above.

\section{Notation, definitions and basic results}\vlabel{preliminaries}

\para{\bf Notation.} If a pro-$p$ group $G$ continuously acts on a profinite space $X$ we call $X$ a $G$-space.   
$H_1(G)$ denotes the first homology $H_1(G,\ffi_p)$ and is canonically isomorphic to
$G/\Phi(G)$. 
If $x\in T$ and $g\in G$, then $G_{gx}=gG_xg\inv$. We shall use the
notation  $h^g=g\inv hg$ for conjugation.  If $H$ a subgroup of $G$,  $H^G$  will stand for the (topological) normal closure 
 of $H$ in $G$. If $G$ is an abstract group $\widehat G$ will mean the pro-$p$ completion of $G$.
  
\para{\bf Conventions.} Throughout the paper, unless otherwise stated, groups are pro-$p$,
subgroups will be closed and 
morphisms will be continuous. Finite graphs of groups will be proper and
reduced (see Definitions \ref{proper} and \ref{reduced}). Actions of
a pro-$p$ group $G$  on a profinite graph $\Gamma$ will a priori be supposed to be faithful
(i.e., the action has no kernel), unless we consider actions on
subgraphs of $\Gamma$.

 \bigskip
 Next  we collect basic definitions, following \cite{horizons}.

\subsection{Profinite graphs}

\begin{defn} A
{\em profinite graph} is a triple $(\Gamma, d_0, d_1)$, where
$\Gamma$ is a profinite (i.e. boolean) space and $d_0,d_1:\Gamma \to \Gamma $ are
continuous maps such that $d_id_j=d_j$ for $i, j \in \{0, 1 \}$.
The elements of $V(\Gamma):=d_0(G)\cup d_1(G)$ are called the
{\em vertices} of $\Gamma$ and the elements of
$E(\Gamma):=\Gamma\setminus V(\Gamma)$ are called the {\em edges} of
$\Gamma$. If $e\in E(\Gamma)$, then $d_0(e)$ and $d_1(e)$ are
called the initial and terminal vertices of $e$. If there is no
confusion, one can just write $\Gamma$ instead of $(\Gamma, d_0,
d_1)$. \end{defn}

\begin{defn}
A {\em morphism} $f:\Gamma \to \Delta$  of graphs is a map $f$ which
commutes with the $d_i$'s. Thus it will send vertices to vertices, but
might send an edge to a vertex.\footnote{It is called a {\em
    quasimorphism} in \cite{R 2017}.} 
\end{defn}

\begin{defn} Every profinite graph $\Gamma$ can be represented as an inverse limit $\Gamma=\varprojlim \Gamma_i$ of its finite quotient graphs (\cite[Proposition 1.5]{horizons}).

A profinite graph $\Gamma$ is said to be {\em connected} if all
its finite quotient graphs are connected. Every profinite graph is
an abstract graph, but in general a connected profinite graph is
not necessarily connected as an abstract graph. \end{defn}

\begin{para} \vlabel{collapse of edges}{\bf Collapsing edges.} If $\Gamma$ is a graph and $e$ an edge which is not a loop we can {\em collapse} the edge $e$ by removing $\{e\}$ from the edge set of $\Gamma$, and
identify $d_0(e)$ and $d_1(e)$ in a new vertex $y$. I.e., $\Gamma'$ is
the graph given by $V(\Gamma')=V(\Gamma)\setminus \{d_0(e),d_1(e)\}\cup
\{y\}$ (where $y$ is a new vertex), and $E(\Gamma')=E(\Gamma)\setminus
\{e\}$. We define $\pi:\Gamma\to \Gamma'$ by setting $\pi(m)=m$ if
$m\notin \{e,d_0(e),d_1(e)\}$, $\pi(e)=\pi(d_0(e))=\pi(d_1(e))=y$. The
maps $d'_i:\Gamma'\to \Gamma'$ are defined so that $\pi$ is a morphism
of graphs. Another way of describing $\Gamma'$ is that
$\Gamma'=\Gamma/\Delta$, where $\Delta$ is the subgraph
$\{e,d_0(e),d_1(e)\}$ collapsed into the vertex $y$.
\end{para}

\subsection{Pro-$p$ trees}\label{pro-p tree}
\begin{para} {\bf An exact sequence.} Let $\Gamma$ be a profinite graph, with set of vertices $V(\Gamma)$ and
$E(\Gamma)=\Gamma\setminus V(\Gamma)$. 
Let $(E^*(\Gamma), *)=(\Gamma/V(\Gamma), *)$ be the pointed profinite
quotient space with $V(\Gamma)$ as  distinguished point, and let
$\mathbb{F}_p[[E^*(\Gamma), *]]$ and $\mathbb{F}_p[[V(\Gamma)]]$ be
respectively the free profinite $\mathbb{F}_p$-modules over the pointed
profinite space $(E^*(\Gamma), *)$ and over the profinite space
$V(\Gamma)$ (cf. \cite[section 5.2]{RZ-10}). Note that when
  $E(\Gamma)$ is closed, then $\mathbb{F}_p[[E^*(\Gamma), *]]=\mathbb{F}_p[[E(\Gamma)]]$. Let the maps $\delta: \mathbb{F}_p[[E^*(\Gamma), *]] \to \mathbb{F}_p[[V(\Gamma)]] $ and $\epsilon : \mathbb{F}_p[[V(\Gamma)]] \to \mathbb{F}_p$ be defined respectively by $\delta(e)=d_1(e)-d_0(e)$ for all $e\in E^*(\Gamma)$ and $\epsilon(v)=1$ for all $v\in V(\Gamma)$. Then we have the following complex of free profinite $\mathbb{F}_p$-modules
\begin{equation*}
    \begin{CD}
      0 @>>> \mathbb{F}_p[[E^*(\Gamma), *]] @>\delta>> \mathbb{F}_p[[V(\Gamma)]] @>\epsilon>>
      \mathbb{F}_p @>>> 0.
    \end{CD}
  \end{equation*}
  \end{para}

\begin{defn}  The profinite graph $\Gamma$ is a {\em pro-$p$ tree} if
  the above sequence is exact. If $T$ is a pro-$p$ tree, then we say
  that a pro-$p$ group $G$ {\em acts on $T$} if it acts continuously on
  $T$ and the action commutes with $d_0$ and $d_1$. We say that $G$ acts
  {\em irreducibly} on $T$ if $T$ does not have proper $G$-invariant
  subtrees and that it acts {\em faithfully}  if the kernel of the action is trivial.   If $t\in V(T)\cup E(T)$ we denote by $G_t$ the stabilizer of $t$ in $G$.\\
For a pro-$p$ group $G$ acting on a pro-$p$ tree $T$ we let $\tilde{G}$
denote the subgroup generated by all vertex stabilizers. Moreover, for
any two vertices $v$ and $w$ of $T$ we let $[v,w]$ denote the geodesic connecting $v$ to $w$ in $T$, i.e., the (unique) smallest pro-$p$ subtree of $T$ that contains $v$ and $w$.
\end{defn}

\subsection{Finite graphs of pro-$p$ groups} 
When we say that ${\cal G}$ is a finite graph of pro-$p$ groups we mean that it contains the data of the
underlying finite graph, the edge pro-$p$ groups, the vertex pro-$p$
groups and the attaching continuous maps. More precisely,

\begin{defn}
let $\Gamma$ be a connected finite graph. A    graph of pro-$p$ groups $({\cal G},\Gamma)$ over
$\Gamma$ consists of  specifying a pro-$p$ group ${\cal G}(m)$ for each $m\in \Gamma$, and continuous monomorphisms
$\partial_i: {\cal G}(e)\longrightarrow {\cal G}(d_i(e))$ for each edge
$e\in E(\Gamma)$.
\end{defn}

\begin{defn}
  \begin{enumerate}
\item A {\em morphism} of graphs of pro-$p$ groups:
$(\G,\Gamma) \rightarrow (\H,\Delta)$ is a pair 
$(\alpha,\bar\alpha)$  of maps, with  
 $\alpha:\G\longrightarrow\H$ a continuous map, and  $\bar\alpha:\Gamma\longrightarrow
\Delta$ a morphism of graphs, and such that $\alpha_{\G(m)}:\G(m)\longrightarrow
\H(\bar\alpha(m))$ is a homomorphism for each $m\in \Gamma$ and  which commutes with the appropriate $\partial _i$. Thus the diagram

$$\xymatrix{
\G\ar@{->}^\alpha[rr]\ar@{->}^{\partial_i}[d] & &\H\ar@{->}^{\partial_i}[d]\\
\G\ar@{->}^{\alpha}[rr] & &\H }$$ is commutative.

\item We say that $(\alpha, \bar\alpha)$ is a {\em monomorphism} if both $\alpha,\bar\alpha$ are injective. In this case its image will be called a subgraph of groups of $(\H, \Delta)$. In other words, a {\em subgraph of groups} of  a graph of pro-$p$-groups
  $(\G,\Gamma)$ is a graph of groups $(\H,\Delta)$, where $\Delta$ is a
subgraph of $\Gamma$ (i.e., $E(\Delta)\subseteq E(\Gamma)$ and
$V(\Delta)\subseteq V(\Gamma)$, the maps $d_i$ on $\Delta$ are the
restrictions of the maps $d_i$ on $\Gamma$), and for each $m\in\Delta,$
$\H(m)\leq \G(m)$.
\end{enumerate}
\end{defn}

\begin{para}{\bf Definition of the fundamental group.} The pro-$p$ fundamental group
$$G= \Pi_1({\cal G},\Gamma)$$
of the graph of pro-$p$ groups $({\cal G},\Gamma)$ is defined by means
of a universal property: $G$ is a  pro-$p$ group together
with the following data  and conditions:
\begin{enumerate}
\item [(i)] a maximal subtree $D$ of $\Gamma$;

\smallskip
\item [(ii)]  a collection of continuous homomorphisms
$$\nu_m: {\cal G}(m)\longrightarrow G\quad (m\in \Gamma), $$
  and     a continuous  map
   $E(\Gamma) \longrightarrow  G$, denoted $e\mapsto t_e$  ($e\in E(\Gamma)$), such that
$t_e=1$ if $e\in E(D)$, and
$$(\nu_{d_0 (e)}\partial_0)(x)= t_e(\nu_{d_1 (e)}\partial_1)(x)t_e^{-1},\quad  \forall x\in {\cal G}(e), \ e\in E(\Gamma); $$

\smallskip
\item [(iii)]  the following universal property is satisfied:

\medskip
\noindent whenever one has the following data

\begin{itemize}
\item $H$ is a pro-$p$ group,\\
\item $\beta_m: {\cal G}(m)\longrightarrow H$, $ (m\in \Gamma)$, 
a collection of continuous homomorphisms,\\
\item a map $e\mapsto s_e$ ($e\in E(\Gamma)$)  with $s_e=1$ if
$e\in E(D)$, and\\
\item $(\beta_{d_0 (e)}\partial_0)(x)= s_e(\beta_{d_1
(e)}\partial_1)(x)s_e^{-1}, \forall x\in {\cal G}(e), \ e\in
E(\Gamma),  $\end{itemize}

\smallskip
\noindent then there exists a unique continuous homomorphism $\delta : G\longrightarrow  H$ such that $\delta(t_e)= s_e$
 $(e\in E(\Gamma))$, and for each $m\in\Gamma$ the diagram

\medskip

$$\xymatrix{&
G  \ar[dd]^\delta   \\  {\cal G}(m)  \ar[ru]^{\nu_m}
\ar[rd]_{\beta_m }\\ &H }$$

\medskip
\noindent commutes.
\end{enumerate}
\end{para}

The main examples of $\Pi_1(\G,\Gamma)$ are an amalgamated free pro-$p$
product $G_1\amalg_H G_2$ and an HNN-extension ${\rm HNN}(G,H,t)$ that correspond to the case of $\Gamma$ having one edge and two and one vertex respectively. 

\begin{defn} \vlabel{proper}We call the graph of groups $(\G,\Gamma)$ {\em proper}
  (injective in the terminology of \cite{R 2017}) if the natural map 
  $\G(v)\to \Pi_1(\G,\Gamma)$ is an embedding for all $v\in V(\Gamma)$.
  \end{defn}

\begin{rem} In the pro-$p$ case, a graph of groups $(\G,\Gamma)$ is not always {proper}. However, the
vertex and edge groups can always be replaced by their images  in
$\Pi_1(\G, \Gamma)$ so that $(\G,\Gamma)$ becomes proper and  $\Pi_1(\G,
\Gamma)$ does not change. Thus through out the paper we shall only
consider  proper graphs of pro-$p$ groups. In particular, all our free amalgamated pro-$p$ products are proper. \end{rem}

If $(\G,\Gamma)$ is a finite graph of
finitely generated pro-$p$ groups, then by a theorem of J-P.~Serre (stating that every finite index subgroup of a finitely generated pro-$p$ group is open, cf. \cite[\S 4.8]{RZ-10}) the  fundamental
pro-$p$ group $G=\Pi_1(\G,\Gamma)$ of
$(\G,\Gamma)$ is the pro-$p$ completion of the usual
fundamental group $\pi_1(\G,\Gamma)$ 
(cf. \cite[\S5.1]{Serre-1980}). Note that $(\G,\Gamma)$ is proper if and only if  $\pi_1(\G,\Gamma)$  is  residually $p$. In particular, edge and vertex groups will be
subgroups of $\Pi_1(\G,\Gamma)$.

\begin{para}{\bf Presentation of the fundamental group.}

In \cite[paragraph (3.3)]{Z-M 89b},  the fundamental group
 $G$ is  defined explicitly in terms of generators and relations
 associated to a chosen subtree $D$. Namely 
 
 \begin{equation} \vlabel{presentation} G=\langle
 \G(v), t_e\mid v\in V(\Gamma), e\in E(\Gamma), t_e=1 \ {\rm for}\  e\in D, \partial_0(g)=t_e\partial_1(g)t_e^{-1},\  {\rm for}\ g\in \G(e)\rangle
\end{equation}
I.e., if one takes the abstract fundamental group $G_0=\pi_1(\G,\Gamma)$,
then $\Pi_1(\G,\Gamma)=\varprojlim_N G_0/N$, where $N$ ranges over
all normal subgroups of $G_0$ of index a power of $p$ and with $N\cap
\G(v)$ open in $\G(v)$ for all $v\in V(\Gamma)$. Note that this last
condition is automatic if $\G(v)$ is finitely generated (as a
pro-$p$-group). 
   It is also proved in \cite{Z-M 89b}
that the definition given above is independent of the choice of
the maximal subtree $D$.\end{para}

\begin{defn}\vlabel{reduced} A finite graph of  pro-$p$ groups $(\G,\Gamma)$ is said to be {\em
reduced}, if for every  edge $e$ which is not a loop,
neither $\partial_{1}(e)\colon \G(e)\to \G(d_1(e))$ nor
$\partial_{0}(e):\G(e)\to \G(d_0(e))$ is an isomorphism. \end{defn}

\begin{rem} \vlabel{reduced-2} Any finite graph of  pro-$p$ groups
can be transformed into a reduced finite graph of pro-$p$ groups by the
following procedure: If $\{e\}$ is an edge which is not a
loop and for which
one of $\partial_0$, $\partial_1$ is an isomorphism,  we can collapse
$\{e\}$ to a vertex $y$ (as explained in \ref{collapse of edges}). Let
$\Gamma^\prime$ be the finite graph given by
$V(\Gamma^\prime)=\{y\}\sqcup V(\Gamma)\setminus\{d_0(e),d_1(e)\}$
and $E(\Gamma^\prime)=E(\Gamma)\setminus\{e\}$, and let
$(\G^\prime, \Gamma^\prime)$ denote the finite graph of  groups
based on $\Gamma^\prime$ given by $\G^\prime(y)=\G(d_1(e))$ if
$\partial_{0}(e)$ is an isomorphism, and $\G^\prime(y)=\G(d_0(e))$ if
$\partial_{0}(e)$ is not an isomorphism. \\
This procedure can be
continued until $\partial_{0}(e), \partial_{1}(e)$ are not surjective for
all edges not defining loops. 
 Note that
the  reduction process does not change the
fundamental pro-$p$ group, i.e., one has a canonical isomorphism
$\Pi_1(\G,\Gamma)\simeq \Pi_1(\G_{red},\Gamma_{red})$.
So, if the pro-$p$ group $G$ is the fundamental group of a finite
graph of pro-$p$ groups, we may assume that the finite graph of
pro-$p$ groups is reduced. \\[0.05in]

\end{rem}

\begin{para}{\bf Standard (universal) pro-$p$ tree.}
Associated with the finite graph of pro-$p$ groups $({\cal G}, \Gamma)$ there is
a corresponding  {\em  standard pro-$p$ tree}  (or universal covering graph)
  $T=T(G)=\bigcupdot_{m\in \Gamma}
G/\G(m)$ (cf. \cite[Proposition 3.8]{Z-M 89b}).  The vertices of
$T$ are those cosets of the form
$g\G(v)$, with $v\in V(\Gamma)$
and $g\in G$; its edges are the cosets of the form $g\G(e)$, with $e\in
E(\Gamma)$; and the incidence maps of $T$ are given by the formulas:

$$d_0 (g\G(e))= g\G(d_0(e)); \quad  d_1(g\G(e))=gt_e\G(d_1(e)) \ \ 
(e\in E(\Gamma), t_e=1\hbox{ if }e\in D).  $$

 There is a natural  continuous action of
 $G$ on $T$, and clearly $ G\backslash T= \Gamma$. There is a standard
connected transversal $s:\Gamma\to T$, given by $m\mapsto \G(m)$.  Note
that  $s_{|D}$  is an isomorphism of graphs and the elements $t_e$ satisfy the  equality $d_1(s(e))=t_es(d_1(e))$. Using the map $s$, we shall identify $\G(m)$ with the stabilizer $G_{s(m)}$ for $m\in \Gamma$:

\begin{equation}\vlabel{transversal}
\G(e)=G_{s(e)}=G_{d_0(s(e))}\cap G_{d_1(s(e))}=\G(d_0(e))\cap t_e\G(d_1(e))t_e^{-1} 
\end{equation}
with $t_e=1$ if $e\in D$. 
Remark also that since $\Gamma$ is finite, $E(T)$ is compact.

\end{para}

\begin{para}{\bf The fundamental group of a profinite graph.}\label{fundamental}
If all vertex and edge groups are trivial we get the definition of the
pro-$p$ {\em fundamental group} $\pi_1(\Gamma)$. It follows that
$\pi_1(\Gamma)$ is a free pro-$p$ group on the base $\Gamma\setminus D$
and so coincides with the pro-$p$ completion $\widehat
\pi_1^{abs}(\Gamma)$ of the abstract (usual) fundamental group
$\pi_1^{abs}(\Gamma)$ that also can be defined traditionally by closed
circuits. \\
Therefore  if $\Gamma$ is  connected profinite and $\Gamma=\varprojlim
\Gamma_i$ is an inverse limit of finite graphs it induces the inverse
system  $\{\pi_1(\Gamma_i)=\widehat \pi_1^{abs}(\Gamma_i)\}$ and
$\pi_1(\Gamma)$ is defined as $\pi_1(\Gamma)=\varprojlim_i
\pi_1(\Gamma_i)$ in this case. The fundamental group $\pi_1(\Gamma)$ acts freely on a pro-$p$ tree $\widetilde \Gamma$ (universal cover) such that $\pi_1(\Gamma)\backslash \tilde\Gamma=\Gamma$ (see  \cite{Z-89} or \cite[Chapter 3]{R 2017} for details).

\end{para}

We shall use frequently in the paper the following known results.

\begin{prop}\label{minimal subtree}(\cite[Lemma 3.11]{horizons}).  Let $G$ be a pro-$p$ group acting on a pro-$p$ tree $T$. Then there exists a nonempty minimal $G$-invariant subtree $D$ of $T$. Moreover, if $G$ does not stabilize a vertex, then $D$ is unique.

\end{prop}

\begin{thm}\label{finite group} (\cite[Theorem 3.9]{horizons})
Let $G$ be a finite $p$-group acting on a pro-$p$ tree $T$. Then $G$ fixes a vertex of $T$.
\end{thm}

\begin{thm}\label{trivial stabilizers}  (\cite[Proposition 2.4]{HZZ} or \cite[Theorem 9.6.1]{R 2017})
Let $G$ be a pro-$p$ group acting on a second countable (as a topological space) pro-$p$ tree $T$ with trivial edge stabilizers. Then there exists a continuous section $\sigma:G\backslash V(T)\longrightarrow V(T)$ and 

$$G=\coprod_{v\in G\backslash V(T)} G_{\sigma(v)}\amalg F,$$ where  $F $ is a free pro-$p$ group naturally isomorphic to $G/\tilde G$.

\end{thm} 

\begin{thm}(\cite[Theorem 7.1.2]{R 2017}, \cite[Theorem 3.10]{Z-M 89b})\label{finite subgroups}
Let $G=\Pi_1(\G, \Gamma)$ be the fundamental pro-$p$ group of a finite
graph of pro-$p$ groups $(\G, \Gamma)$. Then any finite subgroup $K$ of
$G$ is conjugate into some vertex group $\G(v)$. In particular, if the
groups $\G(v)$ are finite, they are exactly the maximal finite subgroups of $G$ up to conjugation. 

\end{thm}

\section{Preliminaries: Auxiliary results}

In this  section we shall prove several auxiliary results on profinite graphs and pro-$p$ groups acting on trees  needed later  in the paper. 

\begin{lem}\vlabel{maximal subtree} (cf. \cite[Corollary 2]{Serre-1980}) Let $\nu:\Delta\longrightarrow
  \Gamma$ be a morphism of finite connected graphs representing the
  collapse of an edge, not a loop. If $T$ is a maximal subtree of
  $\Gamma$, then  $\nu^{-1}(T)$ is a maximal subtree of $\Delta$.\end{lem}

\begin{proof} Consider $\nu^{-1}(T)$ . Since $V(\Gamma)\subset T$, $V(\Delta)\subset \nu^{-1}(T)$. Since $\nu^{-1}(T)$ contains a collapsed edge, $|E(\nu^{-1}(T))|=|E(T)|+1$ and $\nu^{-1}(T)$ is connected. Thus $|E(\nu^{-1}(T))|=|E(T)|+1=|V(\Gamma)|=|V(\Delta)|-1$. Since   $\nu^{-1}(T)$ is connected, it must be a tree, as needed. 

\end{proof} 

\begin{lem}\vlabel{diameter} Let $\Gamma$ be a profinite graph and
  $\Delta$  an abstract connected subgraph of finite diameter $n$
  (i.e. the shortest path between any two vertices has length at most
  $n$). Then the closure $\overline \Delta$ of $\Delta$ in $\Gamma$ has diameter at most $n$.\end{lem}

\begin{proof} Write $\Gamma=\varprojlim \Gamma_i$ as an inverse limit of
  finite quotient graphs and let $\Delta_i$ be the image of $\Delta$ in $\Gamma_i$. Then $\Delta_i$ is finite and has diameter not more than $n$. Since $\overline \Delta=\varprojlim \Delta_i$, so does $\overline \Delta$. Indeed, pick two vertices $v,w$ in $\overline\Delta$ and let $v_i, w_i$ their images in $\Delta_i$. The set $\Omega_i$ of paths of length at most $n$ between $v_i$ and $w_i$ is finite and non-empty. Then $\Omega=\varprojlim \Omega_i$ consists of paths between $v$ and $w$ of length not greater than $n$ and is non-empty.

\end{proof}






\begin{prop} \label{finite diameter} Let $\Gamma$ be a  connected profinite graph of finite diameter.  If $\pi_1(\Gamma)$ is finitely generated, then  there exists a finite connected subgraph $\Delta$ of $\Gamma$ such that $\pi_1(\Gamma)=\pi_1(\Delta)$.

\end{prop}

\begin{proof} By   $\Gamma$ is connected as an abstract graph. Then   by \cite[Proposition 2.7]{SnZ} $\pi_1(\Gamma)$ is the pro-$p$ completion of the usual fundamental group $\pi_1^{abs}(\Gamma)$ and so $\pi_1^{abs}(\Gamma)$ is a free group of the same rank $n$ as $\pi_1(\Gamma)$. Let $D$ be an abstract maximal subtree of the abstract graph $\Gamma$. Then $|\Gamma\setminus D|< \infty$. Let $e_1, \ldots, e_n$ be all edges from $\Gamma \setminus D$. Let $\Omega$ be a minimal subtree of $D$ containing all vertices of $e_1, \ldots, e_n$. Then $\pi_1(\Gamma))$ is free pro-$p$ of rank $n$. Since $\Gamma$ has finite diameter,  $\Omega$ is finite. Therefore $\Delta=\Omega\cup e_1 \cup\cdots \cup e_n$ is a finite connected subgraph of $\Gamma$ and $\pi_1^{abs}(\Delta)$ is a free group of rank $n$. But the fundamental group of a subgraph is a free factor of the fundamental group of a graph, so $\pi_1^{abs}(\Delta)=\pi_1^{abs}(\Gamma)$ and so their pro-$p$ completions  $\pi_1(\Delta)=\pi_1(\Gamma)$.    

\end{proof} 

\begin{prop}\label{modulo stabilizers} Let $G$ be a pro-$p$ group acting
  on a pro-$p$ tree $T$. Then $G/\tilde G=\pi_1(G\backslash T)$ is a free pro-$p$ group acting freely on $\tilde G\backslash T$. Moreover, if $G\backslash T$ is finite,  then the rank of $\pi_1(G\backslash T)$  is $|E(G\backslash T)|-|V(G\backslash T)|+1$. 

\end{prop}

\begin{proof} Recall that $\tilde G$ is the closed subgroup of $G$ generated
  by the vertex stabilisers $G_v$, $v\in T$. By \cite[Corollary 3.9.3]{R 2017} $G/\tilde
  G=\pi_1(G\backslash T)$ is free pro-$p$ and by Proposition
  \cite[Proposition 3.5]{horizons} $\tilde G\backslash T$ is a pro-$p$
  tree. If $\Gamma:=G\backslash T$ is finite it has a maximal subtree $D$ and by \cite[Theorem 3.7.4]{R 2017} a basis of $\pi_1(G\backslash T)$ is $\Gamma\setminus D$. Since $V(D)=V(G\backslash T)$ the result follows.  

\end{proof}

\begin{prop}\label{finite diameter quotient} Let $G$ be a finitely generated pro-$p$ group acting on a pro-$p$ tree $T$ such that $\Gamma=G\backslash T$ has finite diameter. Then $T$ possesses a $G$-invariant subtree $D$ such that $G\backslash D$ is finite.

\end{prop}

\begin{proof} By Proposition \ref{modulo stabilizers}, $G=\tilde G\rtimes \pi_1(\Gamma)$.  We first show that there are finitely many vertices $w_1, \ldots, w_n$ such that $G=\langle G_{w_i}, \pi_1(\Gamma)\mid i=1, \ldots, n\rangle$.

Indeed, let $f:G\longrightarrow G/\Phi(G)$ be the  natural epimorphism
to the quotient modulo the  Frattini subgroup. Then $G/\Phi(G)=f(\tilde
G) \oplus f(\pi_1(\Gamma))$ and since $f(\tilde G)$ is finite (as $G/\Phi(G)$
is) there are vertices $w_1, \ldots, w_n$ of $T$ such that $f(\tilde G)=\langle f(G_{w_1}), \ldots, f(G_{w_n}) \rangle$. Hence $G=\langle G_{w_i}, \pi_1(\Gamma)\mid i=1, \ldots, n\rangle$.

Now since $G$ is fnitely generated, so is $\pi_1(\Gamma)$ and therefore by Proposition \ref{finite diameter}, $\Gamma$ contains a finite
subgraph $\Delta$  such that $\pi_1(\Delta)=\pi_1(\Gamma)$. Let $v_1,
\ldots, v_n$ be the images of $w_1, \ldots, w_n$ in $\Gamma$ and
$\Omega$ a minimal connected graph containing $\Delta$ and $v_1, \ldots,
v_n$. Clearly (because $\Gamma$ has finite diameter) $\Omega$ is finite and so there exists a connected transversal $\Sigma$ of $\Omega$ in $T$.  Let $w'_1, \ldots, w'_n$ be the vertices of $\Sigma$ whose images in $\Omega$ are $v_1, \ldots v_n$ respectively.  Since for each $i$ we have $w'_i=g_iw_i$ for some $g_i\in G$ and so $G_{w'_i}$ is a conjugate of $G_{w_i}$ in $G$, it follows that  $G=\langle G_{w'_i}, \pi_1(\Gamma)\mid i=1, \ldots n\rangle$. Let $D$ be the connected component of the inverse image of $\Omega$ in $T$ containing $\Sigma$. We show that $D$ is $G$-invariant. Let $H=Stab(D)$ be the setwise stabilizer of $D$ in $G$. Clearly, $ G_{w'_i}\leq H$ for each $i$. By \cite[Lemma 2.14]{CZ},  we have
$H \backslash D= \Omega$. Note that $\Delta\subseteq \Omega\subseteq \Gamma$ and so $\pi_1(\Delta)=\pi_1(\Omega)=\pi_1(\Gamma)$.  By Proposition \ref{modulo stabilizers} $H=\tilde H\rtimes \pi_1(\Omega)$, i.e. we may assume that $\pi_1(\Gamma)=\pi_1(\Omega)$ is contained in $H$. But then  $G=H$ and $G\backslash D=\Omega$ is finite as desired.

\end{proof}

\begin{lem} \vlabel{bounding border} Let $G=\Pi_1(\G,\Gamma)$ be the
  fundamental group of a finite graph of pro-$p$ groups
  $(\G,\Gamma)$. Let $D$ be a maximal subtree of $\Gamma$,  $n$ be the
  number of pending vertices of $D$. Then $n \leq 3d(G)$, where $d(G)$
  is  the minimal number of generators of $G$, and $|\Gamma\setminus
  D|\leq d(G)$. 

\end{lem}

\begin{proof} For every pending vertex $v$ of $\Gamma$ and the
  (unique) edge $e$ connected to it, $\overline\G(v)=\G(v)/\G(e)^{\G(v)}$
  is non-trivial, because the graph of groups $(\G, \Gamma)$ is reduced, and the groups are
  pro-$p$. Define the quotient graph of groups $(\overline \G, \Gamma)$ by putting $\overline  \G(m)=1$ if $m\in\Gamma$ is not pending vertex and $\overline\G(v)=\G(v)/\G(e)^{\G(v)}\neq 1$ if $v$ is a pending vertex.  Then from the presentation \eqref{presentation} for
  $\Pi_1(\overline \G, \Gamma)$ it follows then  that $\Pi_1(\overline \G,
  \Gamma)=\coprod_{v\in V(\Gamma)} \overline\G(v) \amalg
  \pi_1(\Gamma)$. The natural morphism $(\G,\Gamma)\longrightarrow (\overline \G,
  \Gamma)$ induces then the epimorphism $G=\Pi_1(\G,\Gamma)\longrightarrow \overline G=\Pi_1(\overline \G,
  \Gamma)$.

   The number of pending vertices of $\Gamma$ is at most 
  $d(\overline G)\leq d(G)$. On the other hand the rank of $\pi_1(\Gamma)$ equals
  $|\Gamma\setminus D|$ and is not greater than $d(G)$ by Proposition
  \ref{modulo stabilizers}.  Every edge of $\Gamma\setminus D$ connects
  at most two pending vertices of $D$ and so the number of  pending vertices of $D$ is at most $3d(\overline G)\leq 3d(G)$. 

\end{proof}

\begin{lem} \vlabel{embedding} Let $G$ be a pro-$p$ group acting on a pro-$p$ tree $T$ with $|G\backslash T|<\infty$. Let $H$ be a subgroup of $G$ with an $H$-invariant subtree $D$ of $T$ such that the natural map $H\backslash D\longrightarrow G\backslash T$ is injective. Then   $G=\Pi_1(\G,G\backslash T)$, $H=\Pi_1(\H, H\backslash D)$  and $(\H, H\backslash D)$ is a  subgraph of groups  of $(\G, G\backslash T)$.  \end{lem}

\begin{proof}   A maximal subtree of $H\backslash D$ can be extended to
  a maximal subtree of $G\backslash T$ and so we can choose a connected
  transversal $S$ of $H\backslash D$  in $D$ that  extends  to a
  connected transversal $\Sigma$ of $G\backslash T$ in $T$. {We may
    further suppose that if an edge $e$ is in $S$ or $\Sigma$, then so
    is $d_0(e)$. } Let $\rho:T\longrightarrow G\backslash T$ be the
  natural epimorphism.

  Then we can define the graph of groups $(\H,H\backslash D )$
 and $(\G,G\backslash T )$  in the standard manner, as follows. If $s\in
 S$, define $\H(\rho(s))=H_s$; if $e\in S$ is an edge, define $\partial_0:\H_{\rho(e)}\to
\H_{\rho(d_0(e))}$ to be the natural inclusion $H_e\to H_{d_0(e)}$, and
if $k_ed_1(e)\in S$, define $\partial_1:\H_{\rho(e)}\to
\H_{\rho(d_1(e))}$ to be the natural
 inclusion $H_e\to H_{d_1(e)}$ followed by conjugation by $k_e\inv$:
$ H_{\rho(d_1(e))} \to H_{\rho(k_ed_1(e))}$. The definition is similar for
 $(\G,G\backslash T )$. 
 
By \cite[Proposition 3.10.4 and Theorem 6.6.1]{R 2017}, we then have $G=\Pi_1(\G,G\backslash T)$, $H=\Pi_1(\H, H\backslash D)$.

\end{proof}

\section{Splitting of pro-$p$ groups acting on trees}

\begin{lem}\vlabel{inverse limit decomposition} Let $G$ be a finitely
  generated pro-$p$ group acting  on a pro-$p$ tree $T$. Then $G=\varprojlim_{U\triangleleft_o G} G/\tilde U$ and $G/\tilde U=\Pi_1(\G_U,\Gamma_U)$ is the fundamental group of a finite reduced graph of  finite $p$-groups. Moreover, the inverse system $\{G/\tilde U, \pi_{VU}\}$ can be chosen in such a way that it is linearly ordered and for each $\{G/\tilde V\}$ of the system with $V\leq U$  there exists a natural morphism $(\eta_{VU},\nu_{VU}):(\G_V, \Gamma_V)\longrightarrow (\G_U, \Gamma_U)$ where $\nu_{VU}$ is just a collapse of edges of $\Gamma_V$ and $\eta_{VU}(\G_V(m))=\pi_{VU}(\G_V(m))$; the induced  homomorphism of the
pro-$p$ fundamental groups coincides with the canonical projection
$\pi_{VU}\colon G/\tilde V\longrightarrow G/\tilde U$.

\end{lem}

\begin{proof} Recall that $\tilde U$ is the closed subgroup of $G$ generated
  by the vertex stabilisers $U_v$. Clearly $G/\tilde U$ and $U/\tilde U$ act on $\tilde U\backslash T$; by Proposition \ref{modulo stabilizers} $U/\tilde U$ is free pro-$p$. Thus
$G_U:=G/\tilde U$ is virtually free pro-$p$.

By \cite[Theorem 1.1]{HZ13}  it follows that $G_U$ is the fundamental pro-$p$
group $\Pi_1(\mathcal{G}_U, \Gamma_U)$ of a finite graph of finite
$p$-groups. As mentioned in Section \ref{preliminaries}  we may assume that $(\G_U,\Gamma_U)$ is reduced.
 
Although  the finite graph of finite 
$p$-groups $(\G_U,\Gamma_U)$ is not uniquely determined by $U$, the index $U$ in the notation
shall express that  these objects are depending on $U$.
Since the maximal finite subgroups of $G_U$ are exactly the vertex
groups of $(\mathcal{G}_U, \Gamma_U)$ up to conjugation (see
Theorem \ref{finite subgroups}), the number of vertices of $\Gamma_U$ does not depend on the choice of $(\G_U,\Gamma_U)$, and since $\pi_1(\Gamma_U)=U/\tilde U$ is free pro-$p$
of rank $|E(\Gamma_U)|-| V(\Gamma_U)|+1$, the size of $\Gamma_U$ is bounded in terms of possible
decompositions as a reduced finite graph of finite $p$-groups of
 $ U/\tilde U$. 

Clearly  we have
$G=\varprojlim_{U}G_U$.

By \cite[Prop.~1.10]{RZ-04}, viewing $G_U$ as a quotient of $G_V$ when
$V\leq U$ (via the natural map $\pi_{VU}\colon G/\tilde V\longrightarrow G/\tilde U$), one has a natural decomposition of
$G/\tilde U$ as the pro-$p$ fundamental group $G/\tilde U=\Pi_1(\G_{VU},\Gamma_V)$
of a finite graph of finite $p$-groups $(\G_{VU},\Gamma_V)$, where the vertex and edge groups satisfy
$\G_{VU}(x)= \pi_{VU}(\G_V(x))$,  
$x\in V(\Gamma_V)\sqcup E(\Gamma_V)$. Thus we have a
morphism $\eta_{VU}\colon (\G_V,\Gamma_V)\longrightarrow (\G_{VU},\Gamma_V)$
of graphs of groups such that the induced homomorphism on the
pro-$p$ fundamental groups coincides with the canonical projection
$\pi_{VU}$. 

\smallskip
If $(\G_{VU},\Gamma_V)$ is not reduced, then collapsing some fictitious
edges $e_i, i=1,\ldots ,k$, we arrive at a reduced graph of groups
$((\G_{VU})_{red},\Delta_V)$. By \cite[Corollary 3.3]{WZ},  the number of
isomorphism classes of finite reduced graphs of finite $p$-groups
$(\G^\prime,\Delta)$ which are based on a finite graph $\Delta$ and satisfy 
$G/\tilde U\simeq\Pi_1(\G^\prime,\Delta)$ is finite.

Using this remark, for each open normal subgroup  $U$ we let $\Omega_U$ be the
(finite) set of reduced finite graphs of finite $p$-groups
$(\G_U,\Gamma_U)$  with $G/\tilde U\simeq\Pi_1(\G_U, \Gamma_U)$.  Let
$V_i$, $i\in \N$, be a decreasing chain of open normal subgroups of $G$
 with $V_0=U$ and $\bigcap_i V_i=(1)$. For $X\subseteq
\Omega_{V_i}$ define $T(X)$ to be the set of all reduced graphs of
groups in $\Omega_{V_{i-1}}$ that can be obtained from graphs of groups
in $X$ by the procedure explained in the preceding paragraph (note that
$T$ does not define a map on $X$). Define $\Omega_1=T(\Omega_{V_1})$,
$\Omega_2=T(T(\Omega_{V_2}))$, \dots,  $\Omega_i=T^{(i)}(\Omega_{V_i})$ and  note
that $\Omega_i$  is a non-empty subset of $\Omega_U$ for every $i\in \N$. Clearly $\Omega_{i+1}\subseteq \Omega_i$ and  since $\Omega_U$  is finite there is an $i_1\in \N$ such that $\Omega_{j}= \Omega_{i_1}$ for all 
$j>i_1$ and we denote this $\Omega_{i_1}$ by $\Sigma_U$. Then
$T(\Sigma_{V_i})=\Sigma_{V_{i-1}}$ for all $i$,  and so we can construct an infinite sequence of graphs of groups $(\G_{V_j}, \Gamma_j)\in\Omega_{V_j}$ such that  
$(\G_{V_{j-1}}, \Gamma_{j-1})\in T(\G_{V_j}, \Gamma_j)$ for all
$j$. This means that $(\G_{V_jV_{j-1}}, \Gamma_{V_j})$ can be reduced to
$(\G_{V_{j-1}}, \Gamma_{j-1})$, i.e., that this sequence $\{(\G_{V_j}, \Gamma_j)\}$ is an inverse system of reduced graphs of groups satisfying the required conditions.

\end{proof} 

Note that in the classical Bass-Serre theory, a finitely generated group $G$ acting  irreducibly on a tree $T$ has finitely many orbits, i.e. $G\backslash T$ is finite. This is not the case in the pro-$p$ case; this fact highlights the complementary difficulties that appear in the pro-$p$ case. The  next result partially overcomes this. 

\begin{thm}\vlabel{splitting} Let $G$ be a finitely generated pro-$p$ group acting  on a pro-$p$ tree $T$ without  global fixed points. Then $G$ splits non-trivially as a free amalgamated pro-$p$ product or pro-$p$ HNN-extension over some stabiliser of an edge of $T$.\end{thm}

\begin{proof} By Lemma \ref{inverse limit decomposition}
  $G=\varprojlim_{U\triangleleft_o G} G/\tilde U$, where $G/\tilde
  U=\Pi_1(\G_U,\Gamma_U)$ is the fundamental group of a finite reduced
  graph of  finite $p$-groups and for each $V\triangleleft_o G$
  contained in $U$, one has a natural morphism $(\eta_{VU},\nu_{VU}):(\G_V, \Gamma_V)\longrightarrow (\G_U, \Gamma_U)$ such that $\nu_{VU}$ is just a collapse of edges of $\Gamma_V$. Moreover, the induced  homomorphism of the
pro-$p$ fundamental groups coincides with the canonical projection
$\pi_{VU}\colon G/\tilde V\longrightarrow G/\tilde U$. 

Note that $U/\tilde U$ is non-trivial for some $U$, since otherwise $G/\tilde U$ is finite for every $U$ and  by Theorem \ref{finite group} stabilizes a vertex $v_U$; hence by an inverse limit argument $G$ would stabilize a vertex $v$ in $T$ contradicting the hypothesis. Hence
 $\Gamma_U$ contains at least one edge.

\smallskip
Case 1. There exists $U$ and an  edge $e_U$ in $\Gamma_U$ such that $\Gamma_U\setminus \{e_U\}$ is disconnected.  

Let $e_V$ be an edge of $\Gamma_V$ such that $\nu_{VU}(e_V)=e_U$. Since
$\Gamma_U$ is obtained from $\Gamma_V$ by collapsing edges,
$\Gamma_V\setminus \{e_V\}$ is disconnected as well. Thus we may write
$G_V=A_V\amalg_{\G_V(e_V)} B_V$, where $A_V$, $B_V$ are the fundamental
groups of the graphs of groups $(\G_V,\Gamma_V)$ restricted to the connected components of $\Gamma_V\setminus \{e_V\}$, and we have an inverse limit of free amalgamated products that gives a decomposition $G=A\amalg_{\G(e)} B$ for some $e\in E(T)$, with $A= \varprojlim_V A_V$, $B=\varprojlim_V B_V$.

\smallskip
Case 2.  For all $U$ and each edge $e_U$ of $\Gamma_U$ the graph
$\Gamma_U\setminus \{e_U\}$ is connected. 

By Lemma  \ref{maximal subtree} (applied inductively) the preimage in
$\Gamma_V$ of a maximal subtree $D_U$ of $\Gamma_U$ is a maximal subtree
$D_V$ of $\Gamma_V$. Therefore for each $V$ we have  $$G/\tilde V={\rm
  HNN}(L_V,\G_V(e), t_e, e\in \Gamma_V\setminus D_V),$$ where
$L_V=\Pi_1(\G_V, D_V)$. Note that the image of $\tilde G$ in $G/\tilde
V$ is $\widetilde{G/\tilde V}$ and  since $G$ is finitely generated, so
is $G/\tilde G$. Therefore, by Proposition \ref{modulo stabilizers},
$\pi_1(\Gamma_V)=F(\Gamma_V\setminus D_V)$ is a free pro-$p$ group of
rank  $|\Gamma_V\setminus D_V|\leq rank(G/\tilde G)$, i.e. we can assume that $\Gamma_V\setminus D_V$ is a constant set $E$.  Then we can view $E$ as a finite subset of $E(T)$ and putting $L=\varprojlim_V L_V$ we have  $G={\rm HNN}(L, G_e, t_e, e\in E)$ for some $e\in E(T)$ as required.

\end{proof}

\begin{cor}\vlabel{lower bound} With the hypotheses of Theorem \ref{splitting}, if $G/\tilde U=\Pi_1(\G_U, \Gamma_U)$ is
  the fundamental group of a reduced finite graph of finite $p$-groups
  as in Lemma \ref{inverse limit decomposition}, then $G$ splits as the pro-$p$ fundamental group of a reduced finite graph of pro-$p$ groups $G=\Pi_1(\G, \Gamma_U)$ with edge groups being stabilizers of some edges  of $T$.

\end{cor}

\begin{proof} We use induction on the size of $\Gamma_U$. Let
  $\pi_U:G\longrightarrow G/\tilde U$ be the natural projection. Pick
  $e_U\in E(\Gamma_U)$. If $\Gamma_U\setminus \{e_U\}=\Delta \cupdot
  \Omega$ is disconnected with two connected components $\Delta$ and
  $\Omega$, then  from the proof of Theorem \ref{splitting} it follows that
  $G$ splits as an amalgamated free product $A\amalg_{G_e}B$ with
  $\pi_U(G_e)=\G_U(e_U)$, where  $\pi_U(A)$ and $ \pi_U(B)$ are the fundamental groups of reduced graphs of groups $(\G_U, \Delta)$ and $(\G_U, \Omega)$ that are restrictions of $(\G_U, \Gamma_U)$ to these connected components. Hence from the induction hypothesis $A=\Pi_1(\G,\Delta)$, $B=\Pi_1(\G, \Omega)$ and the result follows.

 If $\Gamma_U\setminus \{e_U\}$ is connected then again from  the proof
 of Theorem \ref{splitting} it follows that $G_U$ splits as an
 HNN-extension $G_U={\rm HNN}(L,\G_U(e), t_e, e\in \Gamma_U\setminus
 D_U)$, where $D_U$ is a maximal subtree of $\Gamma_U\setminus \{e_U\}$
 and 
 $\pi_U(G_e)=\G_U(e_U)$, $\pi_U(L)=\pi_1(\G_U, D_U)$. Then by induction hypothesis $L=\Pi_1(\G, D_U)$  and $G=\Pi_1(\G, \Gamma_U)$ as needed.
 
 Finally we observe that $(\G,\Gamma_U)$ is reduced since $(\G_U, \Gamma_U)$ is.

\end{proof}

A.A. Korenev \cite{korenev} defined the number of pro-$p$ ends $e(G)$ for an infinite pro-$p$ group $G$ as $e(G)=1+dim H^1(G, \F_p[[G]])$. The next theorem shows that similar to the abstract case a pro-$p$ group acting irreducibly on an infinite  pro-$p$ tree with finite edge stabilizers has more than one end. 

\begin{thm}\label{number of ends}  Let $G$ be a finitely generated pro-$p$ group acting   on a pro-$p$ tree $T$ with finite edge stabilizers and without   global fixed points.   Then either $G$ is virtually cyclic and $H^1(G, \F_p[[G]])\cong \F_p$ (i.e. $G$ has two ends) or  $H^1(G, \F_p[[G]])$ is infinite  (i.e. $G$ has infinitely many ends).\end{thm}

\begin{proof} By Theorem \ref{splitting} $G$ splits either as an
  amalgamated free pro-$p$ product or an HNN-extension over an edge
  stabilizer  $G_e$ and so acts on the standard pro-$p$ tree $T(G)$
  associated with this splitting. Let $H$ be an open normal subgroup of $G$  intersecting
  $G_e$ trivially. Then $H$ acts on $T(G)$ with trivial edge stabilizers and so by Theorem \ref{trivial stabilizers}  $H$ is a non-trivial free pro-$p$ product
  $H=H_1\amalg H_2$. 

Then we have the following exact sequence (associated to the standard pro-$p$ tree) for this free product decomposition:

$$0\rightarrow \F_p[[H]]\xrightarrow{\delta}{} \F_p[[H/H_1]]\oplus \F_p[[H/H_2]]\xrightarrow{\epsilon} \F_p\rightarrow 0 \eqno(*)$$

\medskip
{\bf Claim.} The augmentation ideal $I(H)$ is decomposable as an $\F_p[[H]]$-module.

\smallskip
Proof. Let $M_1$ and $M_2$ be the kernels of the restrictions of
$\epsilon$  to  $\F_p[[H/H_1]]$ and $\F_p[[H/H_2]]$ respectively. We
will show that $\delta(I(H))=M_1\oplus M_2$. Since
$\delta(\F_p[[H]])=\ker(\epsilon)$, $M_1\oplus M_2$ is a submodule of
$\delta(\F_p[[H]])$ and since the middle term of $(*)$ modulo $M_1\oplus
M_2$ is $\F_p\oplus \F_p$, it is of index $p$ in $\ker(\epsilon)$. But $\F_p[[H]]$ is a local ring and so has a unique maximal left ideal, hence $\delta(I(H))=M_1\oplus M_2$ as needed. The claim is proved.

\bigskip

 Now applying $Hom_{\F_p[[H]]}(-,\F_p[[H]])$ to 
 $$0\rightarrow I(H)\rightarrow \F_p[[H]]\rightarrow \F_p\rightarrow 0 $$
and observing that by \cite[Lemma
 3]{korenev}$$Hom_{\F_p[[H]]}(\F_p,\F_p[[H]])=(\F_p[[H]])^H =0,\ \  Hom_{\F_p[[H]]}(\F_p[[H]],\F_p[[H]])=\F_p[[H]]$$ and $Ext^1_{\F_p[[H]]}(\F_p[[H]],M)=0$ since $\F_p[[H]]$ is a free pro-$p$ module, we obtain the exact sequence

$$0\rightarrow \F_p[[H]]\xrightarrow{\varphi}{} Hom_{\F_p[[H]]}(I(H), \F_p[[H]])\rightarrow H^1(H,\F_p[[H]])\rightarrow 0.$$

(Here we also use that $Ext^1_{\F_p[[H]]}(\F_p, M)=H^1(H,M)$ for an
$\F_p[[H]]$-module $M$). Since $\F_p[[H]]$ is indecomposable and $$Hom_{\F_p[[H]]}(I(H),
\F_p[[H]])\cong Hom_{\F_p[[H]]}(M_1, \F_p[[H]])\oplus
Hom_{\F_p[[H]]}(M_2, \F_p[[H]])$$ (from the claim), $\varphi$ is not onto and so $H^1(H,\F_p[[H]])\neq 0$.

Then by \cite[Theorems 1,2]{korenev}, the dimension of $H^1(H,\F_p[[H]])$ is either infinite or 1 and in the latter case $H$ is virtually cyclic. By \cite[Lemma 2]{korenev}, $H^1(H,\F_p[[H]])\cong H^1(G,\F_p[[G]])$, hence the result.

\end{proof}  
  
\section{Subgroups of fundamental groups of graphs of pro-$p$ groups}  

In  the classical Bass-Serre theory of groups acting on trees  a finitely generated group $G$ acting on a tree $T$ is the fundamental group of
a finite graph of groups whose edge and vertex groups are $G$-stabilizers of
edges and vertices of $T$ respectively. This is due to the fact that for finitely generated $G$ there exists a $G$-invariant subtree $D$ such that $G\backslash D$ is finite.  In the pro-$p$ situation this is not always the case. Note that $G\backslash D$  finite implies that there are only finitely many maximal stabilizers of vertices of $T$ in $G$ up to conjugation. In
this section we prove  a result mentioned above in the pro-$p$ case  under the assumption of finitely many maximal vertex stabilizers up to conjugation.

\begin{thm}\vlabel{General} 
Let $G$ be a finitely generated  pro-$p$ group acting
on a pro-$p$ tree $T$ with finitely many maximal vertex stabilisers up to conjugation.
 Then  $G$ is  the fundamental group of a reduced finite graph
of  pro-$p$ groups $(\G,\Gamma)$, where  each vertex
group $\G(v)$ and each edge group $\G(e)$ is  a maximal vertex stabilizer $G_{\tilde v}$ and  an edge stabilizer
$G_{\tilde e}$ respectively (for some $\tilde v,\tilde e\in T$). 

\end{thm}

\begin{proof}   By Lemma \ref{inverse limit decomposition}, $G=\varprojlim_{U\triangleleft_o G} G/\tilde U$, where $G/\tilde U=\Pi_1(\G_U,\Gamma_U)$ is the fundamental group of a finite reduced graph of  finite $p$-groups and for each $V\leq U$ one has a natural morphism $(\eta_{VU},\nu_{VU}):(\G_V, \Gamma_V)\longrightarrow (\G_U, \Gamma_U)$ such that $\nu$ is just a collapse of edges of $\Gamma_V$. Moreover, the induced  homomorphism of the
pro-$p$ fundamental groups coincides with the canonical projection
$\pi_{VU}\colon G/\tilde V\longrightarrow G/\tilde U$.

We claim now that the number of vertices and edges of $\Gamma_U$ is
bounded independently of $U$. Let $G_{v_1}, \ldots, G_{v_n}$ be the maximal vertex
stabilizers of $G$ up to conjugation. Then $G_U$ and $U/\tilde U$ act on $\tilde U\backslash T$; by Proposition \ref{modulo stabilizers}, the quotient group $U/\tilde U$ acts freely on the
pro-$p$ tree $\tilde U\backslash T$. Thus all vertex stabilizers of
$G/\tilde U$ are finite and are the images of the corresponding vertex stabilizers of $G$. Note that any finite subgroup of $G/\tilde U$ stabilizes  a vertex (Theorem \ref{finite group}) and since the maximal finite subgroups of $G_U$ are exactly the vertex
groups of $(\mathcal{G}_U, \Gamma_U)$ up to conjugation (see
Theorem \ref{finite subgroups}), we see that the number of vertices of $\Gamma_U$ is bounded by $n$. Since $\pi_1(\Gamma_U)\cong G/\tilde G$ is free pro-$p$ 
of rank $|E(\Gamma_U)|-|V(\Gamma_U)|+1$ by Proposition \ref{modulo stabilizers}, the number of edges of  $\Gamma_U$ is bounded by $n+d(G)-1$ and so the size of $\Gamma_U$ is bounded independently of $U$. 
Hence, for some $U$ and all $V\leq U$,  the maps $\nu_{VU}:\Gamma_V\to
\Gamma_U$ are isomorphisms, and we will denote this graph by $\Gamma$.

Then for
$(\G,\Gamma)=\varprojlim\, (\G_U, \Gamma) $ we have $\G(x)=\varprojlim
\G_U(x)$ if $x$ is either a vertex or an edge of $\Gamma$, and $(\G,\Gamma)$
is a reduced finite graph of  pro-$p$ groups
satisfying $G\simeq\Pi_1(\G,\Gamma)$. This finishes the proof of the theorem.

\end{proof}

\medskip\noindent
\begin{cor}\vlabel{finiteness of maximal orbits}  The number up to
  conjugation of maximal vertex stabilizers in $G$  
 equals $|V(\Gamma)|$.\end{cor} 

\begin{proof}  Since $(\G,\Gamma)=\varprojlim_U (\G_U,\Gamma)$, the result follows from Theorem \ref{General}.
   
  \end{proof}

  One of the main consequences of the main theorem of Bass-Serre theory
  is an extension of the Kurosh subgroup theorem to a group $G$ acting
  on tree $T$. Namely if  $H$ is a subgroup of $G$ then
  $H=\pi_1(\H,\Delta)$ is the fundamental group of a graph of groups constructed as follows.  Let $\Delta=H\backslash T$ and  if $\Sigma$ is a connected transversal of $\Delta$ in $T$ then $\H$ consists of stabilizers of the edges and vertices of $\Sigma$. 
 
 In the pro-$p$ situation such a theorem does not hold in general
 (\cite[Theorem 1.2]{HZ archiv}). Our next objective is to prove it
 for $H$ acting acylindrically and having finitely many maximal vertex stabilizers up to conjugation.

\begin{defn}
The action of a pro-$p$ group $G$ on a pro-$p$ tree $T$ is said to be
\emph{$k$-acylindrical}, for $k$ a constant, if for every $g\neq 1$ in $G$,
the subtree $T^g$ of fixed points  has diameter at most $k$. 
\end{defn}

\begin{thm}\vlabel{acylindrical} Let $G$ be a finitely generated  pro-$p$ group acting $n$-acylindrically
on a pro-$p$ tree $T$ with finitely many maximal vertex stabilizers up to conjugation. Then

\begin{enumerate}

\item[(i)] The closure $\overline D$ of $D=\{t\in T\mid G_t\neq 1\}$ is a profinite $G$-invariant subgraph of $T$ having finitely many connected components $\Sigma_i$, $i=1,\ldots, m$ up to translation. 

\item[(ii)] for the  setwise stabilizer      $G_i=Stab_G(\Sigma_i)$ the quotient graph $G_i\backslash \Sigma_i$ has finite diameter and $\Sigma_i$ contains  a  $G_i$-invariant subtree $D_i$ such that $G_i\backslash D_i$ is finite.

\item[(iii)]   $G=\coprod_{i=1}^m G_i\amalg F$ is a free pro-$p$ product, where $F$ is a free pro-$p$ group acting freely on $T$.

\end{enumerate}
\end{thm} 

\begin{proof}
We follow the idea of the proof of \cite[Theorem 3.5]{SnZ}.

Since the action is  $n$-acylindrical,
$ T^{G_t}$ has diameter at most  $n$  for every non-trivial edge or vertex
stabilizer $G_t$.  Note that
$D=\bigcup_{G_t\neq 1} T^{G_t}$. We show  that  $G\backslash D$
has finite diameter (as an abstract graph). Indeed, since there are only finitely many
maximal vertex stabilizers up to conjugation, say $G_{v_1}, \ldots G_{v_k}$,  it suffices to show
that for a maximal vertex stabilizer $G_{v_i}$, the tree $\bigcup_{1\neq
  G_t\leq G_{v_i}} T^{G_t}$ has finite
diameter (if non-empty). But for $1\neq G_t\leq G_{v_i}$ the geodesic $[t,v_i]$ is
stabilized by $G_t$ (cf. \cite[Corollary 3.8]{horizons}) and so has
length not more than $n$. 
Thus $G\backslash D$ as an abstract graph has finite diameter (not more
than $2nk$) and finitely many connected components (not more than
$k$). 

It follows that the closure $\Delta$ of $G\backslash D$ in $G\backslash T$ has also  finitely many (profinite) connected components (not more than $k$) and    finite  diameter (not greater than $2nk$)
(see Lemma \ref{diameter}). Note that the preimage of $\Delta$ in $T$ is exactly $\overline D$. Since $\Delta$ has finite diameter it is immediate that connected components of $\overline D$ are mapped surjectively onto corresponding connected components of $\Delta$, thus the number of connected components of $\overline D$ up to translation equals the number of connected components of $\Delta$ ($\leq k$). This proves (i). 

\medskip

Collapsing
all connected components of $\overline D$, by
Proposition on Page 486 of \cite{Zalesskii89} or \cite[Proposition
3.9.1, as well as Cor.~3.10.2 and Prop.~3.10.4]{R 2017}, we get a pro-$p$ tree $\bar T$ on which
$G$ acts with trivial edge stabilizers and vertex stabilizers being
the setwise stabilizers $G_i=Stab_G(\Sigma_i)$ of the connected components
$\Sigma_i$ of $\bar D$. In particular, we have only finitely many
vertices $v'_1, \ldots, v'_m$ up to translation whose stabilizers are
non-trivial  (and $m\leq k$). So by Theorem \ref{trivial stabilizers}  $G$ is a
free pro-$p$ product
$$G=\coprod_{i=1}^m G_{v'_i}\amalg  F,$$
where $F$ is naturally isomorphic to $G/\tilde G$ with $\tilde G$ taken
with respect to the action on $\bar T$ and for each $i$,  $G_{i}$ is the setwise
stabilizer $Stab_G(\Sigma_i)$ of some connected component $\Sigma_i$ of
$\bar D$ that was collapsed to $v'_i$. Then $F$ acts freely on $T$ and  (iii) is proved.

\medskip

 By \cite[Lemma 2.14]{CZ}, for any connected component
$\Sigma_i$ of $\overline D$  and its
setwise stabilizer $G_i=Stab_G(\Sigma_i)$ we have
$G_i \backslash\Sigma_i\subseteq \Delta$ and so $G_i \backslash\Sigma_i$ is a connected component $\Delta_i$ of $\Delta$. So $\Delta_i$ has finite diameter and 
by Proposition \ref{finite diameter quotient} $\Sigma_i$ possesses a  $G_i$-invariant subtree $D_i$ such that $G_i\backslash D_i$ is finite. This proves (ii).
\end{proof}

\begin{cor}\vlabel{virtually subgraph of groups} Let $G$ be a finitely
  generated pro-$p$ group which is the fundamental group
  of a finite graph $(\G,\Gamma)$ of pro-$p$ groups, and let $T=T(G)$ its
  standard pro-$p$ tree. Let $H$ be   a  finitely generated subgroup of
  $G$  that acts $n$-acylindrically on $T$, with finitely many maximal
  vertex stabilizers up to conjugation.
Then $H=\coprod_{i=1}^m H_i \amalg F$ (possibly with one factor) with $F$ free pro-$p$ and  there exists an open
subgroup $U$ of $G$ containing $H$ such that 

\begin{enumerate}
\item[(i)] The natural map $F\longrightarrow U/\tilde U$ is  injective.

\item[(ii)] 
 $U=\Pi_1(\U,U\backslash
T)$, $H_i=\Pi_1(\H_i, H_i\backslash D_i)$ (where $D_i$ is a minimal
$H_i$-invariant subtree of $T$) and $(\H_i, H_i\backslash D_i)$ are
disjoint  subgraphs of groups  of $(\U, U\backslash T)$. 
\end{enumerate}

Moreover, the
latter statements (i) and (ii) hold for any open subgroup $V$ of $U$
containing the group $H$.\end{cor}

\begin{proof}  By Theorem \ref{acylindrical} applied to the action of
  $H$ on $T$, there are subgroups $H_i$ ($i=1,\ldots,k$) and $F$ of $H$, with $F$ free
  pro-$p$, such that $H=\coprod_{i=1}^m H_i\amalg F$. Furthermore
  there are $H_i$-invariant subtrees $D_i$ of $T$ with $H_i\backslash
  D_i$ finite, and $H_i=Stab_H(D_i)$, and $F=H/\langle H_1,\ldots,H_m\rangle^H$. Note that the
  $H_i\backslash D_i$ are disjoint subgraphs of $H\backslash T$, and are
  contained in $\bar D$ (notation as in Theorem \ref{acylindrical}). Choose an
  open subgroup $U$ of $G$ containing $H$ such that  the map
  $\bigcup_i(H_i\backslash D_i)\to U\backslash T$ is injective and the
  map $F\to U/\tilde U$ is injective (this is possible
  since $\bigcup _i H_i\backslash D_i$ is finite, $F$ is finitely
  generated and $U/\tilde U$ is free pro-$p$). Then the
  $H_i\backslash D_i$ are disjoint in $U\backslash T$, whence if we
  choose maximal subtrees $T_i$ of $H_i\backslash D_i$, their union
  extends to a maximal subtree of $U\backslash T$. As $G\backslash T$ is
  finite, so is $U\backslash T$, and we can then  apply (the proof of)
  Lemma 
  \ref{embedding} to get the result.

\end{proof}

\section{Generalized accessible pro-$p$ groups}\label{Generalized Accessible}

 We apply here the results of the previous section  to finitely generated generalized accessible 
pro-$p$ groups in the sense of the following definition.

\begin{defn} Let $\cF$ be a family of pro-$p$ groups. A pro-$p$ group $G$ will be called 
  $\cF$-{\em accessible} if there is a number $n=n(G)$ such that any finite,
  proper, reduced graph of pro-$p$ groups  with  edge groups in $\cF$
  having fundamental group isomorphic to $ G$ has at most $n$
  edges.\end{defn}
  
  The definition generalizes the definition of accessibility given in \cite{wilkes}, where the edge groups are finite. In fact if $\cF$ is the class of all finite $p$-groups, an $\cF$-accessible pro-$p$ group will simply be called {\em accessible}.

  \begin{prop}\vlabel{restricting maximal stabilizers} Let $\cF$ be a
    family of pro-$p$ groups and $G$ a finitely generated $\cF$-accessible pro-$p$ group
    acting on a pro-$p$ tree $T$ with edge stabilizers in $\cF$. Then $G$ has only finitely many maximal vertex stabilizers up to conjugation, and in fact the number of such stabilizers does not exceed the $\cF$-accessibility number $n(G)$. \end{prop}

\begin{proof} Let 
  $G_{v_1}, \ldots G_{v_m}$ be  maximal vertex stabilizers which are
  non-conjugate. We will show that $m$ is bounded. If
  $U\triangleleft_0G$, then  $G/\tilde U$  acts on
  $\tilde U\backslash T$ and by Lemma \ref{inverse limit decomposition}
  $G=\varprojlim_{U\triangleleft_o G} G/\tilde U$, where $G/\tilde
  U=\Pi_1(\G_U,\Gamma_U)$ is the fundamental group of a finite reduced
  graph of  finite $p$-groups. Thus starting from a certain $U$ the stabilizers $G_{v_1}\tilde U/\tilde U, \ldots G_{v_m}\tilde U/\tilde U$  of the images of $v_1,\ldots, v_m$ in $\tilde U\backslash T$ are still maximal and distinct. So they are maximal finite subgroups of $G/\tilde
  U=\Pi_1(\G_U,\Gamma_U)$   and so are conjugate to vertex groups of $(\G_U,\Gamma_U)$ (see Theorem \ref{finite subgroups}). Therefore $\Gamma_U$ has at least $m$ vertices.  Then by Corollary \ref{lower bound}, $G$ admits a decomposition as the fundamental group of a reduced finite graph of pro-$p$ groups  $\Pi_1(\G,\Gamma_U)$ with edge groups in $\cF$ and so $m\leq n(G)$.  

\end{proof} 

\begin{thm} \vlabel{F-accessible acting on tree}  
Let $G$ be a finitely generated $\cF$-accessible  pro-$p$ group acting
on a pro-$p$ tree $T$ with edge stabilizers in $\cF$.
 Then  $G$ is  the fundamental pro-$p$ group of a finite graph
of finitely generated pro-$p$ groups $(\G,\Gamma)$, where  each vertex
group $\G(v)$ and each edge group $\G(e)$ is  a vertex stabilizer $G_{\tilde v}$ and  an edge stabilizer
$G_{\tilde e}$ respectively (for some $\tilde v,\tilde e\in T$). Moreover, the size of $V(\Gamma)$ is bounded by the accessibility number $n(G)$ for every such $(\G,\Gamma)$.

\end{thm}

\begin{proof} By Proposition \ref{restricting maximal stabilizers} the number of maximal vertex stabilizers of $G$ is bounded by   the accessibility number $n=n(G)$. Therefore the result follows from Theorem \ref{General} and Corollary \ref{finiteness of maximal orbits}.

\end{proof}

\begin{rem}\vlabel{virtually subgraph of groups 2} In the classical
  Bass-Serre theory of groups acting on trees  structure theorems like
  Theorem \ref{F-accessible acting on tree} are used to obtain structure
  results on subgroups of fundamental groups of graph of groups (see for
  example \cite[\S 5]{Serre-1980}). In our situation, to use Theorem
  \ref{F-accessible acting on tree} for this purpose one needs to assume that $\cF$ is closed under subgroups. A relevant to the context  such general example is the class of small pro-$p$ groups. Namely, one can follow the approach of  \cite{BF} in the abstract case and
 call  a pro-$p$ group $G$ {\em small}  if  whenever $G$ acts on a
pro-$p$-tree $T$, and $K\leq G$ acts freely on $T$, then $K$ is procyclic.  If the action on $T$ is associated with splitting $G$ into  a free amalgamated  product $G=G_1\amalg_H G_2$ or an HNN-extension $G={\rm HNN}(G_1, H, t)$ it means that  $H$ is  normal in $G$ with $G/H$  either procyclic or infinite dihedral.    The class of small pro-$p$ groups $\cals$ is closed under subgroups. Then one can use Theorem \ref{F-accessible acting on tree} to prove the following statement.

\medskip  {\em Let $G$ be a pro-$p$
  group acting on a pro-$p$ tree $T$ with small edge stabilizers. Let $H$ be a finitely generated
  $\cals$-accessible subgroup of $G$. Then  $H=\Pi_1(\H, \Gamma)$ is  the fundamental group of a finite graph
of  pro-$p$ groups $(\G,\Gamma)$, where  each vertex
group $\G(v)$ and each edge group $\G(e)$ is  a vertex stabilizer $G_{\tilde v}$ and  an edge stabilizer
$G_{\tilde e}$ respectively (for some $\tilde v,\tilde e\in T$). Moreover, the size of $\Gamma$ is bounded by the accessibility number $n(H)$.}

\medskip
Note also  that a finitely generated free pro-$p$ group $F$ is $\cals$-accessible, since a free small pro-$p$ group has to be pro-cyclic.

\end{rem}
  
 \medskip 
  \begin{example}\vlabel{accessibility} If $\C$ is a class of pro-cyclic
    pro-$p$ groups then any finitely generated pro-$p$ group is
    $\C$-accessible (\cite[Lemma 3.2]{SnZ}).\end{example}

In fact we can bound the $\C$-accessibility number $n(G)$ in terms of
the minimal number of generators $d(G)$ of $G$.

\begin{prop} \vlabel{lem1}Let $G=\pi_1(\Gamma,\calg)$ be the fundamental group of a
  finite graph of pro-$p$-groups, with procyclic edge groups, and assume
  that $d(G)=d\geq 2$. Then (assuming that the graph is reduced), the vertex groups are finitely generated, the number
  of vertices of $\Gamma$ is $\leq 2d-1$, and the number of edges of $\Gamma$ is
  $\leq 3d-2$. 
\end{prop}

\begin{proof} Let 
$T$ be a maximal subtree of $\Gamma$ and $H=\Pi_1(\G,T)$ be the fundamental group of the tree of groups $(\G,T)$ obtained by restricting $(\G,\Gamma)$ to $T$. Then  $G={\rm HNN}(H, C_1,\ldots,C_\ell,
t_1,\ldots,t_\ell)$, where $\ell=|E(\Gamma)|-|E(T)|$. Note that  the quotient of $G$ by the
normal subgroup generated by $H$ is free on $t_1,\ldots,t_\ell$,   so $d(G)\geq \ell$.

Since $|V(\Gamma)|=|V(T)|$, and $|E(T)|=|V(T)|-1$ we have $|E(\Gamma)|=|V(T)|-1+\ell\leq |V(T)|+d-1$. It
therefore suffices to show that $|V(T)|=|V(\Gamma)|\leq 2d-1$. 

\smallskip
Consider the Mayer-Vietoris sequence
$$\rightarrow \oplus_{e\in E(T)}H_1(\calg(e))\rightarrow \oplus_{v\in V(T)}H_1(\calg(v))\rightarrow H_1(G)\rightarrow
\ffi_p[[E(T)]]\rightarrow \ffi_p[[V(T)]]\rightarrow \ffi_p\rightarrow 0$$
{(see }\cite[Thm 9.4.1]{R 2017}). Since $T$ is a tree, 
$$0\rightarrow
\ffi_p[[E(T)]]\rightarrow \ffi_p[[V(T)]]\rightarrow \ffi_p\rightarrow 0$$ is exact (see Subsection \ref{pro-p tree}) and so
 $H_1(G)\rightarrow
\ffi_p[[E(T)]]$ is the zero map\red{,} $\oplus_{v\in V(T)}H_1(\calg(v))\rightarrow H_1(G)$ is onto.

Let $n$ be the number of vertices of $T$, $m$ be the number of vertices
whose vertex groups are cyclic and $k$  be the number of vertices whose
vertex groups are not cyclic, so that $n=m+k$ and the number of edges of $T$ is
$n-1$. Then 
$$d(G)=dim(H_1(G))\ge  m+2k-(n-1)=m+2k-n+1=k+1$$ 
(here we are using  that all edge groups are cyclic). On the other hand if the vertex group
$\calg(v)$ is cyclic and $e$ is incident to $v$ then the natural map
$H_1(\calg(e))\rightarrow H_1(\calg(v))$ is the zero map (because
$\G(e)\leq \Phi(\G(v))$). Denoting by $V_c$ the set of vertices with cyclic vertex group, it follows that 
$\oplus_{v\in V_c} H_1(\calg(v))$ intersects trivially the image of
$\oplus_{e\in E(T)}H_1(\calg(e))$ and therefore maps injectively into
$H_1(G)$. Therefore $m\leq d$. Thus $n=k+m\leq d(G)-1+d(G)=2d(G)-1$ as
required. 

Finally, since $\oplus_{e\in E(T)}H_1(\calg(e))$ and $H_1(G)
$ are finite, so is $\oplus_{v\in V(T)} H_1(\calg(v))$, i.e. $\calg(v)$ is finitely generated for every $v$.

\end{proof}

\begin{rem} If $H=\Pi_1(\G,T)$ is non-trivial then $|E(\Gamma)|$ is strictly less then $3d-2$ since  $d(G)=dim(G/\Phi(G))=\ell+ dim((H/\Phi(H))/\langle[C_1, t_1], \ldots, [C_\ell,t_\ell]\rangle]\geq \ell+1>\ell$.\end{rem}

Taking into account Example \ref{accessibility} and Proposition \ref{lem1} we deduce

\begin{thm}\vlabel{cyclic General}
Let $G$ be a finitely generated pro-$p$ group acting on a pro-$p$ tree $T$ with procyclic edge stabilizers.
 Then  $G$ is  the fundamental group of a finite graph
of finitely generated pro-$p$ groups $(\G,\Gamma)$, where  each vertex
group $\G(v)$ and each edge group $\G(e)$ is conjugate into a subgroup
of a vertex stabilizer $G_{\tilde v}$ and an edge stabilizer
$G_{\tilde e}$ respectively. Moreover,   $|V(\Gamma)|\leq 2d-1$, and $|E(\Gamma)|\leq 3d-2$ for any such $(\G,\Gamma)$ , where $d$ is the minimal number of generators of $G$.

\end{thm}

We now apply Theorem \ref{cyclic General} to the pro-$p$ analogue of a limit groups defined in \cite{Kochloukova2}. 
It is worth recalling their definition. 

Denote by
$\mathcal{G}_0$ the class of all free pro-$p$ groups of finite
rank. We define inductively the class $\mathcal{G}_n$ of pro-$p$
groups $G_n$ in the following way: $G_n$ is a free 
amalgamated pro-$p$ product $G_{n-1}\amalg_{C}A$, where $G_{n-1}$ is any
group from the class $\mathcal{G}_{n-1}$, $C$ is any
self-centralized procyclic pro-$p$ subgroup of $G_{n-1}$ and $A$
is any finite rank free abelian pro-$p$ group such that $C$ is a
direct summand of $A$. The class of pro-$p$ groups $\mathcal{L}$ (pro-$p$ limit groups) 
consists of all finitely generated pro-$p$ subgroups $H$ of some
$G_n\in \mathcal{G}_n$, where $n\geq 0$. 
Then $H$ is a  subgroup of a free
amalgamated pro-$p$ product $G_n=G_{n-1}\amalg_{C}A$, where
$G_{n-1}\in \mathcal{G}_{n-1}$, $C\cong \mathbb{Z}_p$ and
$A=C\times B\cong \mathbb{Z}_p^m$. By Theorem 3.2 in
\cite{Ribes3}, this amalgamated pro-$p$ product is proper. Thus $H$
acts naturally on the pro-$p$ tree $T$ associated to $G_n$  and its edge stabilizers are procyclic.

An immediate application of Theorem \ref{cyclic General} then gives a
bound on the $\C$-accessibility number of a limit pro-$p$ group.

\begin{cor}\vlabel{cor-KS} Let $G$ be a pro-$p$ limit group. Then  $G$ is  the fundamental group of a finite graph
of finitely generated pro-$p$ groups $(\G,\Gamma)$, where   each edge group $\G(e)$ is infinite procyclic. Moreover, $|V(\Gamma)|\leq 2d-1$, and $|E(\Gamma)|\leq 3d-2$, where $d$ is the minimal number of generators of $G$.

\end{cor} 
  
\subsection*{Accessible pro-$p$ groups}

\medskip This subsection is dedicated to accessible pro-$p$
groups. Note that there exists a finitely generated non-accessible
pro-$p$ group \cite{wilkes} and that it is an open question whether a finitely presented pro-$p$ group is accessible.  
  
  The next proposition gives a characterization of accessible pro-$p$ groups.
  
  \begin{prop} \vlabel{accesibility characterization} Let $G$ be a
    finitely generated pro-$p$ group. Then $G$ is accessible if and only
    if  it is a virtually free pro-$p$ product of finitely many  virtually freely indecomposable  pro-$p$ groups.\end{prop}

\begin{proof} Let $H$ be an open subgroup of $G$ that  splits  as a free pro-$p$ product of virtually freely indecomposable pro-$p$ groups. Replacing $H$ by the core of $H$ in $G$ and applying the
Kurosh subgroup theorem for open subgroups (cf. \cite[Thm.~9.1.9]{RZ-10}),
we may assume that $H$ is normal in $G$. Refining the free decomposition if necessary
and collecting free factors isomorphic to $\Z_p$ we obtain a free decomposition 
\begin{equation}
\vlabel{eq:Hdeco}
H=F\amalg H_1\amalg\cdots\amalg H_s,
\end{equation}
where $F$ is a free subgroup of rank $t$, and the $H_i$ are
virtually freely indecomposable finitely generated subgroups which are not
isomorphic to $\Z_p$.   By \cite[Theorem 3.6]{WZ} $G=\Pi_1(\G, \Gamma)$ is the fundamental pro-$p$ group of a finite graph of pro-$p$ groups with finite edge groups. Moreover, it follows from its proof (step 2) that $H$ intersects all edge groups trivially. Then by \cite[Theorem 3.1]{wilkes} $\Gamma$ has at most $\frac{p[G:H]}{p-1}(d(G)-1)+1$ edges. So $G$ is accessible.

 Conversely, suppose $G$ is accessible. Write $G=\Pi_1(\G, \Gamma)$,
 where $(\Gamma,\G)$ is a finite graph of pro-$p$ groups with finite
 edge groups, and such that $\Gamma$ is of maximal size.  Choosing an open
 normal subgroup $H$ intersecting all edge groups of $G$ trivially we
 have    $$H=\coprod_{v\in V(\Gamma)} \coprod_{g_v\in H\backslash G/
   \G(v)} H\cap \G(v)^{g_v}\amalg F,$$ where  $g_v$ runs through double
 cosets representatives of $H\backslash G/ \G(v)$ and $F$ is free pro-$p$
 of finite rank (see Theorem \ref{trivial stabilizers} with use of the action of $H$ on the standard pro-$p$ tree $T(G)$). Since $\Gamma$ is of maximal size, $\G(v)$ does not
 split as an amalgamated free pro-$p$ product or HNN extension over a
 finite $p$-group, so  by \cite[Theorem A]{WZ} $\G(v)$ is not a virtual free pro-$p$ product, in particular $H\cap  \G(v)^{g_v}$ is freely indecomposable. Since $F$ is a free pro-$p$ product of $\Z_p$'s the result follows.

\end{proof}

\begin{question} Let $G$ be a finitely generated pro-$p$ group acting on a pro-$p$ tree with finite vertex stabilizers. Is $G$ accessible?

\end{question}

Note that $H^1(G,\F_p[[G]])$ is a right $\F_p[[G]]$-module.
The next theorem gives  a sufficient condition of accessibility for  a pro-$p$ group in terms of this module; we do not know whether the converse also holds (it holds in the abstract case).

\begin{thm} \vlabel{accessible-2} Let $G$ be a finitely generated pro-$p$ group. If $H^1(G, \F_p[[G]])$ is  a finitely generated $\F_p[[G]]$-module, then $G$ is accessible.\end{thm}

\begin{proof}  Suppose $G=\Pi_1(\G,\Gamma)$ is the fundamental group of
  a reduced finite graph $(\G,\Gamma)$ of pro-$p$ groups with finite
  edge groups. We will first do the case where if $v$ is any vertex of $\Gamma$,
  then $\G(v)$ is infinite and $H^1(\G(v),\ffi_p[[\G(v)]])\neq 0$. 
The group $G$ acts   on the standard pro-$p$ tree $T$ associated to
  $(\G,\Gamma)$, and we get  
 
 $$ 0\longrightarrow \oplus_{e\in E(\Gamma)} \F_p[[G/\G(e)]]\longrightarrow \oplus_{v\in V(\Gamma)} \F_p[[G/\G(v)]]\longrightarrow \F_p\longrightarrow 0$$
 
  Applying $Hom_{\F_p[[G]]}(-,\F_p[[G]])$ to  this exact sequence  and taking into account that $$Hom_{\F_p[[G]]}(\F_p,\F_p[[G]])=(\F_p[[G]])^G =0$$ (\cite[Lemma 3]{korenev}), we get
  
  $$0\rightarrow \oplus_{v\in V(\Gamma)} Hom_{\F_p[[G]]}(\F_p[[G/\G(v)]], \F_p[[G]])\rightarrow \oplus_{e\in E(\Gamma)}Hom_{\F_p[[G]]}(\F_p[[G/\G(e)]], \F_p[[G]])\rightarrow $$
  $$H^1(G,\F_p[[G]])\rightarrow \oplus_{v\in V(\Gamma)} Ext^1_{\F_p[[G]]}(\F_p[[G/\G(v)]], \F_p[[G]])\rightarrow \oplus_{e\in E(\Gamma)} Ext^1_{\F_p[[G]]}(\F_p[[G/\G(e)]], \F_p[[G]])$$
  
  By Shapiro's lemma $$Hom_{\F_p[[G]]}(\F_p[[G/\G(v)]], \F_p[[G]])=H^0(\G(v), Res_{\F_p[[\G(v)]]}\F_p[[G]])=(\F_p[[G]])^{\G(v)}=0,$$
 the latter equality since $\G(v)$ is infinite (\cite[Lemma
 3]{korenev}); similarly, by Shapiro's lemma,  $$Ext^1_{\F_p[[G]]}(\F_p[[G/\G(v)]], \F_p[[G]])=H^1(\G(v), Res_{\F_p[[\G(v)]]}\F_p[[G]]),$$ $$Hom_{\F_p[[G]]}(\F_p[[G/\G(e)]], \F_p[[G]])=H^0(\G(e), Res_{\F_p[[\G(e)]]}\F_p[[G]]),$$ $$Ext^1_{\F_p[[G]]}(\F_p[[G/\G(e)]], \F_p[[G]])=H^1(\G(e), Res_{\F_p[[\G(e)]]}\F_p[[G]]).$$ 
 
 Now since $\G(e)$ is finite, $G$ has a system of open normal subgroups
 $U$ intersecting $\G(e)$ trivially and so  
$$Res_{\F_p[[G_e]]}\F_p[[G]]=\varprojlim_U Res_{\F_p[[G_e]]}\F_p[[G/U]]=\varprojlim_U (\bigoplus_{\G(e)\backslash G/U} \F_p[\G(e)])= \prod_I \F_p[\G_e]$$ for some infinite set of indices $I$ (see \cite[Corollary 2.3]{MZ}). Moreover, since Hom commutes with projective limits in the second variable    we have  $$H^0(G(e), Res_{\F_p[[G_e]]}\F_p[[G]])= \varprojlim_U (\bigoplus_{\G(e)\backslash G/U} H^0(\G(e), \F_p[[\G(e)]])).$$ But  $H^0(\G(e), \F_p[[\G(e)]])\cong \F_p$.  Thus $$\varprojlim_U (\bigoplus_{\G(e)\backslash G/U} H^0(\G(e), \F_p[[\G(e)]]))=\varprojlim_U (\bigoplus_{\G(e)\backslash G/U} \F_p)=\varprojlim_U \F_p[[\G(e)\backslash G/U]])=\F_p[[\G_e\backslash G]]$$
 
 Note also that Ext commutes with direct products on the second variable  and  
  $ H^1(\G(e), \F_p[[\G(e)]])=0$, since a free $ \F_p[[\G(e)]]$-module is injective.
 So 
  $$H^1(G(e), Res_{\F_p[[G_e]]}\F_p[[G]])=H^1(G(e),  \prod_I \F_p[G_e])=  \prod_I H^1(\G(e), \F_p[[\G(e)]]))=0$$

  Thus the above long exact sequence can be rewritten as 
$$0\longrightarrow  \oplus_{e\in E(\Gamma)} \F_p[[G/\G(e)]]\longrightarrow H^1(G,\F_p[[G]])\longrightarrow \oplus_{v\in V(\Gamma)}H^1(G(v), Res_{\F_p[[\G(v)]]}\F_p[[G]])\longrightarrow 0$$

We show now that $H^1(\G(v), Res_{\F_p[[\G(v)]]}\F_p[[G]])\neq 0$ for each $v$. Indeed, $Res_{\F_p[[\G(v)]]}\F_p[[G]])$ is a free $\F_p[[\G(v)]]$-module (see \cite[Proposition 7.6.3]{W}) and so is projective by \cite[Proposition 7.6.2]{W}. Then since $\F_p[[G]]$ is a local ring by \cite[Proposition 7.5.1]{W} and \cite[Proposition 7.4.1 (b)]{W} $Res_{\F_p[[\G(v)]]}\F_p[[G]])=\prod_{i\in I} \F_p[[\G(v)]] $ is a direct product of copies of $\F_p[[\G(v)]]$. Since $Ext$ commutes with the direct product on the second variable we have $H^1(\G(v), Res_{\F_p[[\G(v)]]}\F_p[[G]])=\prod_{i\in I} H^1(\G(v), \F_p[[\G(v)]])$. But for every $v$  the groups $H^1(\G(v), \F_p[[\G(v)]])\neq 0$ by the assumption at the beginning of the proof. So $H^1(\G(v), Res_{\F_p[[\G(v)]]}\F_p[[G]])\neq 0$ for every $v$. 

Hence the number of vertices in $\Gamma$ cannot exceed  the minimal
number of generators of $\F_p[[G]]$-module $ H^1(G,\F_p[[G]])$. The
number of edges of $\Gamma$ cannot exceed $d(G)+ |V(\Gamma)|-1$ since
the rank of $\pi_1(\Gamma)=G/\tilde G$ equals
$|E(\Gamma)|-|V(\Gamma)|+1$, where the equality $\pi_1(\Gamma)=G/\tilde
G$ follows from  \cite[Corollary 3.9.3]{R 2017} combined with
\cite[Proposition 3.10.4 (b)]{R 2017}.\\

We will now do the general case. First observe that if
$G=\Pi_1(\G,\Gamma)$, where $(\G,\Gamma)$ is a reduced finite graph of
pro-$p$-groups with finite edge groups, then, letting $T$ be a maximal
subtree of $\Gamma$, there are at most $d:=d(G)$ edges in $\Gamma\setminus
T$, and therefore there are at most $3d$ pending vertices in $\Gamma$,
see Lemma \ref{bounding border}. We will now bound the
size of $T$. 

Suppose now that some vertex group $\G(v)$ is either finite or has $H^1(\G(v),\F_p[[G(v)]])=
0$. If $e\in E(T)$ is adjacent to $v$, with other extremity $w$, then
collapsing $\{e,v,w\}$ into a new vertex $y$, and putting on top of $y$
the group $\G(y)=\G(v)\amalg_{\G(e)}\G(w)$, by Theorem \ref{number of
    ends} we have $H^1(\G(y),\ffi_p[[\G(y)]])\neq 0$, and $\G(y)$ is infinite. Let 
$M$ be the number of generators of $H^1(G,\ffi_p[[G]])$.\\

{\bf Claim.} The diameter of $T$ is at most $2M$.

Indeed, if not, it contains a path with $2M+2$ distinct vertices. But
applying the above procedure to get rid of the bad vertices on the path, produces at
least $M+1$ vertices $y$ with  $H^1(\G(y),\ffi_p[[\G(y)]])\neq 0$ and 
$\G(y)$ infinite, which
contradicts the first part. \\

The result now follows, as there is a bound on the size of trees with
at most $3d$ pending vertices and diameter $\leq 2M$, and $|\Gamma\setminus T|\leq d$ (see Proposition \ref{modulo stabilizers}).

\end{proof}

\section{Howson's property}\vlabel{Howson property}

\begin{defn} We say that a pro-$p$ group $G$  {\em has Howson's
    property}, or {\em is Howson}, if whenever $H$ and $K$ are two
   finitely generated closed subgroups of $G$, then $H\cap K$ is
  finitely generated. 
\end{defn}
Free pro-$p$-groups are Howson, and the Howson
  property is preserved under free (pro-$p$) products, see \cite[Thm
  1.9]{SZ}. In this section we investigate the preservation of Howson's
  property under various (free) constructions.   
  \begin{lem}\vlabel{G-action} Let $G$ be a finitely generated pro-$p$
    group acting on a profinite space $Y$ such that the number of  maximal point stabilizers $G_a$, up to conjugation,
      is finite and represented by the elements $a\in A\subseteq Y$. Let $H$ be a subgroup of $G$ such that $H_y$ is finitely generated for each $y\in Y$ and the kernel of the natural homomorphism $\beta:\F_p[[H\backslash Y]])\longrightarrow  \F_p[[G\backslash Y]])$ is finite. Then the image of the natural homomorphism  $\eta: H_1(H,\F_p[[Y]])\longrightarrow H_1(H)$ is finite. \end{lem} 
  
  \begin{proof}  We use the characterisation $H_1(G)=G/\Phi(G)$.   By Shapiro's lemma $H_{\red{1}}(G, \F_p[[G/G_y]])=H_1(G_y)=G_y/\Phi(G_y)$ and so the image of $\gamma: H_1(G, \F_p[[Y]])\longrightarrow H_1(G)$   coincides with the smallest closed
subgroup containing all images of the $H_1(G_y)$'s. Observe now that if $G_y\leq G_a$ and $g\in G$, then $G_y\Phi(G)\leq G_a
\Phi(G)=G_a^g \Phi(G)$, whence the image of $H_1(G_y)$ in $H_1(G)$ is contained in the
image of $H_1(G_a)$, which equals the image of $H_1(G_a^g)$. Thus the
image of $\gamma$ coincides with the subgroup of $H_1(G)$ generated by
the images of $H_1(G_a)$, $a\in A$.

A given $ H$-orbit
$H/H_y\subset Y$ is sent by $\beta$ to a subset of a
$G$-orbit $G/G_y$  in $Y$. If for $y\in Y$ the stabilizer $G_y$ is not
maximal, then there exists a maximal $G_a$, $a \in A$, and $g\in G$,
such that $G_y\leq G_a^g$. Hence $H_y\leq H_a^g$.    Since
$\ker(\beta)$ is finite, the set $B$ of $H$-orbits in $Y$  that
map into some $G$-orbit $Ga$ with $a\in A$,  is finite. Note that $H_1(
H,\ffi_p[[H/H_a]])=H_1(H_a)$ (by Shapiro's lemma). Thus the image of
$\eta: H_1( H,\F_p[[Y]])\longrightarrow H_1(H)$ coincides with the group
generated by the images of $H_1(H_b)$, $b\in B$.  But $H_a$ is finitely
generated, so each $H_1(H_a^g)$ is finite, and since $B$ is finite,  the
image of $\eta$ is also finite. 
  
  \end{proof}



\begin{thm}\vlabel{howson} Let $G$ be a finitely generated pro-$p$ group acting on a pro-$p$ tree $T$  with procyclic edge stabilizers such that $G\backslash T$ is finite.  Let $H$ be  a finitely generated subgroup of $G$ such that $H\backslash T(H)$ is finite, where $T(H)$ is a minimal $H$-invariant subtree of $T$ and $H_{\tilde e}\neq 1 $ for all $e\in E(T(H))$.  If $K$ is a finitely generated  subgroup of $G$ then  $H\cap K$ is finitely generated in each of the following cases:

\begin{enumerate} 
\item[(i)] $K$ intersects trivially all vertex stabilizers $H_v$, $v\in V(T(H))$;
\item[(ii)] the vertex stabilizers  $G_{ v}$ are Howson, $v\in V(T)$.
\end{enumerate} 

\end{thm} 

\begin{proof} The proof follows the idea of the proof of Theorem 1.9
  \cite{SZ}. Put $\Delta=H\backslash T(H)$. Observe that if $u\in
  T(H)$, then  $Hu$ is the
  intersection of all $Uu$ with $U$ an open subgroup of $G$ containing
  $H$. As $H\backslash T(H)$ is finite, there is  some open subgroup $U$ of $G$ containing
  $H$ and  such that we have an injection $H\backslash T(H)\to U\backslash
  T$.  By Lemma \ref{embedding}, passing to an open subgroup of  $G$
  containing $H$,  we
  may therefore assume that $G=\Pi_1(\G, \Gamma)$ with $\Gamma$ finite and that
  $(\H,\Delta)$ is a subgraph of groups of $(\G,\Gamma)$ such that $H=\Pi_1(\H,\Delta)$.
  Moreover, since $\Delta$ is finite, replacing $G$ by an  open subgroup we may assume that $\G(e)=\H(e)$ for every $e\in E(\Delta)$.

By \cite[Lemma 5.6.7]{RZ-10} there exist continuous sections
$\eta:K\backslash T\to T$ and $\kappa:(H\cap K)\backslash T(H)\to T(H)$, and by \cite[beginning of Section 9.4]{R 2017}  we have the following long exact sequences 

\begin{multline} H_1(K,\underset{{v\in K\backslash V(T)}}\oplus \F_p[[G/G_{\eta(v)}]])\to H_1(K,\ffi_p)\to \underset{{e\in K\backslash E(T)}}\oplus \F_p[[K\backslash G/G_{\eta(e)}]]
  \to \\
  \hfill \xrightarrow{\delta} 
  \underset{{v\in K\backslash V(T)}}\oplus \F_p[[K\backslash G/G_{\eta(v)}]]
   \to \F_p\to 0\end{multline}
and

\begin{multline}\vlabel{longexact for intersection}
H_1(K\cap H, \bigoplus_{v\in (K\cap H)\backslash V(T(H))} \ffi_p[[H/H_{\kappa(v)}]])\to H_1(H\cap
  K,\F_p){\longrightarrow} \\
  \hfill \bigoplus_{e\in (K\cap H)\backslash E(T(H))} \F_p[[K\cap H\backslash H/H_{\kappa(e)}]] \xrightarrow{\sigma}
\bigoplus_{v\in (K\cap H)\backslash V(T(H))} \F_p[[(K\cap
H)\backslash H/H_{\kappa(v)}]] \to \F_p\to 0. \end{multline}

We then have the following commutative diagramme:

$$
\begin{diagram}
  \node{ \bigoplus_{e\in K\backslash E(T)} \F_p[[K\backslash G/G_{\eta(e)}]]}\arrow{e,t}{\delta}\node{\bigoplus_{v\in K\backslash V(T)} \ffi_p[[K\backslash G/G_{\eta(v)}]]}\\
  \node{\bigoplus_{e\in (K\cap H)\backslash E(T(H))} \F_p[[K\cap H\backslash H/H_{\kappa(e)}]]}\arrow{e,t}{\si} \arrow{n,t}{\alpha}\node{\bigoplus_{v\in (K\cap H)\backslash V(T(H))} \ffi_p[[K\cap H\backslash H/H_{\kappa(v)}]]}\arrow{n,t}{\beta}  
  \end{diagram}$$
\noindent 

We want to show that $\ker \beta\circ \si$ is finite, or equivalently
that $\ker \delta\circ\alpha$ is finite. The dimension (as an
$\ffi_p$-v.s)  of $\ker \delta$ is $\leq \dim(H_1(K))$, i.e., less than
or equal to
the number of generators of $K$.  So, we need to show that
$\ker(\alpha)$ is finite, and if possible bound its size.
We know that there is an inclusion of $(K\cap H)\backslash H$ in
$K\backslash G$, and we need to see what happens when we quotient by the
action of $\G(e)$ (on the right).\\[0.05in]

The inclusion map $\ffi_p[[(K\cap H)\backslash H]]\to \ffi_p[[K\backslash
    G]]$ is a map of right $\ffi_p[[H_e]]$-modules for any $e\in E(T(H))$, and note that it sends
distinct $H_e$-orbits to distinct $H_e$-orbits (this is where we use that
$\G(e)=\H(e)$ for $e\in E(\Delta)$ and so $G_e=H_e$ for every $e\in E(T(H)) $. Hence  $\alpha$ is an injection!!\\[0.05in]
To summarise: $\delta\circ\alpha = \beta\circ\sigma$ have finite
 kernel, of dimension bounded by $d(K)$. Furthermore, as the
image of $\si$ in $\oplus_{v\in (H\cap K)\backslash V(T(H))} \ffi_p[[K\cap H\backslash
H/H_{\kappa(v)}]]$ has codimension $1$ (by the exact sequence \eqref{longexact for intersection}), it
follows that $\ker(\beta)$ is also finite, and we have
$\dim(\ker(\si))+\dim (\ker (\beta))\leq d(K)+1$.\\[0.05in]

(i) if $K$  intersects trivially all conjugates of $H_v$ then the
 left term  of (\ref{longexact for intersection}) is $$H_1(K\cap H,\bigoplus_{v\in (K\cap H)\backslash V(T(H))} \ffi_p[[H/H_{\kappa(v)}]])$$ and equals  $0$ because $\bigoplus_{v\in (K\cap H)\backslash V(T(H))} \ffi_p[[H/H_{\kappa(v)}]]$ is a free $K\cap H$-module, so (i) follows from the injectivity of $\alpha$.

\medskip
(ii) Since the $\G(v), v\in V(\Delta) $,  are Howson, and $K_v, H_v$ are
finitely generated by Theorem~\ref{cyclic General},  $K\cap H_v$ is
finitely generated for any $v\in V(T)$. By \ref{accessibility} and Corollary~\ref{restricting maximal stabilizers} the set of maximal vertex $K$-stabilizers is finite up to conjugation.  Thus, since $\ker(\beta)$ is finite, we can apply Lemma~\ref{G-action}  to $K\cap H\leq K$ to deduce that the image of  $H_1(K\cap H, \oplus_{v\in (K\cap H)\backslash V(T(H))}\ffi_p[[H/H_{\kappa(v)}]])$ in $H_1(H\cap K)$ is finite.

Combining this with the finiteness of $\ker(\sigma)$ we deduce that $H_1(H\cap K)$ is finite, i.e. $H\cap K$ is finitely generated.
\end{proof} 

\begin{rem}\label{acylindrically} If a finitely generated subgroup $H$
  of $G$ acts $n$-acylindrically on $T$
  and does not split as a free pro-$p$ product, then the hypotheses of
  Theorem \ref{howson} on $H$ are satisfied automatically by Corollary~\ref{virtually subgraph of groups}. In particular, this holds for limit pro-$p$ groups.

\end{rem}

\begin{cor} Let $G=G_1\amalg_C G_2$ be a free amalgamated pro-$p$ product of
  free or Demushkin not soluble pro-$p$ groups with $C$ maximal
  procyclic in $G_1$ or $G_2$. Let $H$ be a finitely generated subgroup of $G$ that does not split as a free pro-$p$ product. Then $H\cap K$ is finitely generated for any finitely generated subgroup $K$ of $G$.\end{cor}  

\begin{proof} In this case the action of $G$ on its standard pro-$p$
  tree $T$ is 2-acylindrical. This follows from the fact that if $C$ is
  maximal procyclic in say $G_1$, then $C$ is malnormal in $G_1$,
  i.e. $C\cap C^{g_1}= 1$ for any $g_1\in G_1\setminus C$. Indeed, if
  $C\cap C^{g_1}\neq 1$ then $\langle C, C^{g_1}\rangle$ normalizes this
  intersection. But every 2-generated subgroup of $G_1$ is free and so
  can not have  procylic normal subgroups, so $C=C^{g_1}$ and so $g_1$
  normalizes $C$. But then the same applies to $\langle C, g_1\rangle$
  so it is procyclic contradicting the maximality of $C$ in $G_1$.

Thus by Remark \ref{acylindrically} we obtain the result.

\end{proof}

\bigskip
\begin{thm} \vlabel{R1} Let $G=G_1\amalg_CG_2$ be a free  pro-$p$ product with procyclic amalgamation.  Let $H_i\leq G_i$,  be finitely generated such that $C\cap H_1\cap H_2\neq 1$ and 
$K\leq G$  a finitely generated subgroup of $G$. Then $K\cap H$ is
  finitely generated in each of the following cases

\begin{enumerate} 
\item[(i)]  $K$  intersects all conjugates of $H_i$ trivially. Moreover, if $C\leq H_i$ ($i=1,2$) then $d(H\cap K)\leq d(K)$.

\item[(ii)]  
The  $G_i$'s are Howson pro-$p$.
 
 \end{enumerate} 
\end{thm}

\begin{proof}  The group $G$ acts on its standard pro-$p$ tree $T=T(G)$
  and so we can apply Theorem~\ref{howson}  to deduce (ii) and the first part of (i). Thus, we only need
to show the second part of (i).

 To obtain the precise bound $d(K)$, we need to show that  the
natural map 
 $H\backslash T(H)\to G\backslash T$ is an injection: but this follows
from our hypothesis $C\leq H_1\cap H_2$ since this map is in fact an isomorphism. 

As was observed in the proof of Theorem \ref{howson}, $\delta\circ\alpha = \beta\circ\sigma$ have finite
dimensional kernel, of dimension bounded by $d(K)$. 

If $K$ does not intersect the conjugates of $H_i$ then the left term of
equation \eqref{longexact for intersection} (in the proof of
\ref{howson}) is $0$, so  from the injectivity of $\alpha$ one deduces that the natural map $H_1(K\cap H)\longrightarrow H_1(K)$ is an injection.
\end{proof}

 \begin{thm}\vlabel{R2} Let $G={\rm HNN}(G_1,C,t)$ be a  pro-$p$ HNN
   extension, $G_1$ a finitely generated  and $C\neq 1$ procyclic. Let
   $H_1$ be a finitely generated subgroup of $G_1$ such that $H_1\cap
   C\neq 1$ and  $H=\langle H_1,t\rangle$. Then for a finitely generated
   subgroup  $K$  of $G$ the intersection $K\cap H$ is finitely
   generated in each of the following cases
 
 \begin{enumerate}
 \item[(i)]  $K$  intersects trivially  every conjugate of $H_1$. Moreover, if $C\leq H_1$  then $d(H\cap K)\leq d(K)$.
 
 \item[(ii)]
   $G_1$  satisfies 
 Howson's property.

\end{enumerate} 
 \end{thm}

\begin{proof} The proof is identical to the proof of Theorem \ref{R1}.
  The group $G$ acts on its standard pro-$p$ tree $T=T(G)$ and so we can
  apply Theorem \ref{howson} to deduce (ii) and the first part of (i).   Therefore we just need to show  the second part of (i).

To obtain the precise bound $d(K)$, we need to show that  the
natural map 
 $H\backslash T(H)\to G\backslash T$ is an injection: but this follows
from our hypothesis since this map is in fact an isomorphism when $C\leq
H_1$.

As was observed in the proof of Theorem \ref{howson}, $\delta\circ\alpha = \beta\circ\sigma$ have finite
dimensional kernel, of dimension bounded by $d(K)$. 

If $K$ does not intersect the conjugates of $H_i$ then the left term of
equation (\ref{longexact for intersection}) 
is $0$, so  from the injectivity of $\alpha$ one deduces that the natural map $H_1(K\cap H)\longrightarrow H_1(K)$ is an injection.

\end{proof}

\section{Normalizers}  \vlabel{Normalizers}

\begin{prop}\vlabel{normalizers} Let $C$ be a procyclic pro-$p$ group and $U\leq C$ a subgroup of $C$.
\begin{enumerate} 

\item[(a)] Let $G=G_1\amalg_CG_2$  and $N=N_G(U)$. Then $N=N_{G_1}(C)\amalg_C N_{G_2}(C)$.

\item[(b)] Let $G={\rm HNN}(G_1, C, t)$ be a proper pro-$p$ HNN-extension. 

(i) If there is some $g\in G_1$ such
that $U^g=U^t$, then $N_G(U)={\rm HNN}(N_{G_1}(U), C, t')$.

(ii) If $U$ and $U^t$ are not conjugate in $G_1$ then $N_G(U):= N:=N_1\amalg_CN_2$,
where $N_1=N_{G_1^{t\inv}}(U)$ and $N_2=N_{G_1}(U)$.

\end{enumerate}

\end{prop} 

\begin{proof} 

The proof follows the proof of
 Proposition 2.5 in \cite{RZ-94} or Proposition 15.2.4 (b) \cite{R 2017} . Let $T$ be the standard pro-$p$ tree for $G$. 
By \cite[Theorem 3.7]{horizons} the subset $Y=T^U$   of
 $T$ consisting of points fixed by
 $U$ is a pro-$p$ subtree. Observe that if
 $g\in G$, then  $U$ fixes  $gC$ if and only if $U^g\leq C$.  \\
 
Then $N$ acts on $Y$ continuously. Indeed, if $g\in
N$, $y\in Y$ and $u\in U$, then $u g=gu'$ for some
$u'\in U$, and therefore $u gy=gu'y=gy$. This being true for all $u$ in $U$, we
get that $N$ acts on $Y$. \\

Consider the natural epimorphism $\varphi:T\to G\backslash T$. 
Then the natural map $\psi: Y \to
N\backslash Y$ is the restriction of $\varphi$ to $Y$.
 To see this pick 
$h\in G$  such that $hC\in E(Y)$; so $U\leq C^h$ and therefore $U,U^{h}\leq C$. As $C$ is procyclic, we get $U=U^h$, i.e., $h\in N$ (work in
finite quotients of $G$ where the equality is obvious). 
This shows that $\psi$  
coincides with the restriction of $\varphi$ to $Y$.

Thus $N\backslash E(Y)$ consists of one edge only, and therefore
$N\backslash Y$ has at most two vertices. According to
Proposition~4.4 in \cite{MZ-90} (or \cite[Theorem 6.6.1]{R 2017}),  we have $N=N_1\amalg_CN_2$, where
$N_1=N_{G_1}(U)$ and $N_2=N_{G_2}(U)$,  or $N={\rm HNN}(N_{G_1}(U),C,
t')$, depending on  whether $Y$ has two vertices or just one vertex. (Note that $N$
contains $C$). In Case (a)
$\varphi(gG_1)\neq \varphi(gG_2)$, so $\psi(Y)$ has two vertices. 

In Case (b) $N\backslash Y$ has one vertex only iff $d_1(C)=tG_1$ is in
the $N$-orbit of $d_0(C)=G_1$, i.e., if 
$G_1^t=G_1^n$ for some $n\in N$ iff $G_1=G_1^{nt\inv}$ iff $g=nt\inv\in G_1$, 
in which case $U^{g}=U^t$ as required.

\end{proof}

\begin{prop}\vlabel{normalizer of cyclic} Let $G$ be a pro-$p$ group acting on a pro-$p$ tree $T$ and  $U$ be a procyclic
     subgroup of $G$ that does not stabilize any edge. Then one of the following happens:
     \begin{enumerate}
      
      \item For some $g\in G$ and vertex $v$, $U\leq G_v$: then $N_G(U)=N_{G_v}(U)$.
\item For all $g\in G$ and vertex $v$, $U\cap G_v=\{1\}$. Then $N_G(U)/K$ is
    either  isomorphic to $\Z_p$ or to the  dihedral group $C_2\amalg C_2=\Z_2\rtimes C_2$, where $K$ is some normal subgroup of $N_G(U)$ contained in the stabilizer of an edge.
        
       \end{enumerate}
     
     \end{prop}
     
      \begin{proof} Let $N=N_G(U)$ and let $D$ be a minimal $U$-invariant subtree of $T$ (that exists by Proposition \ref{minimal subtree}).

Case 1. $|D|=1$, i.e., $U$ stabilizes a vertex $v$. 
If $v\neq nv$ for some $n\in N$, then by Corollary 3.8 in \cite{horizons},
$U$ stabilizes all edges in $[v,nv]$,  contradicting our hypothesis. So $N_G(U)$ fixes $v$ and  we have (1). 

 \smallskip 
   
   Case 2. $D$ is not a vertex. Then $U$ acts irreducibly on $D$ and so
   by  Proposition \ref{minimal subtree} it is unique. Note that if
   $n\in N$, then $nD$ is also $D$-invariant, and therefore must equal
   $D$. 
   Hence $N$ acts irreducibly on $D$ and by Lemma 4.2.6 (c) in \cite{R
     2017} $C_N(U)K/K$ is free pro-$p$, where $K$ is the kernel of the
   action (and is the intersection of all stabilizers). Hence $C_N(U)K/K$ is procyclic (because $UK/K$ is procyclic,
   $\neq 1$) and so $N/K$ is either $\Z_p$ or $C_2\amalg C_2$, since
   $Aut(U)\cong \Z_p\times C_{p-1}$ for $p>2$ or $ \Z_2\times C_{2}$ for
   $p=2$.  
   \end{proof}

Combining Theorem \ref{R1} and Propositions \ref{normalizers} and \ref{normalizer of cyclic} we deduce the following 

 \begin{thm}\vlabel{normalizer-3} Let $C$ be a procyclic pro-$p$ group and $G=G_1\amalg_CG_2$ be a free amalgamated pro-$p$ product or a pro-$p$ HNN-extension $G={\rm HNN}(G_1,C,t)$ of Howson groups. Let $U$ be a procyclic
     subgroup of $G$   and $N=N_G(U)$.  Assume that $N_{G_i}(U^g)$ is finitely
   generated whenever $U^g\leq G_i$. If $K\leq
   G$ is finitely generated, then so is $K\cap N$. 
 \end{thm}
 
 \begin{proof} Let $T$ be the standard pro-$p$ tree for $G$. If $U$
   does not stabilize any edge then by Proposition \ref{normalizer of
     cyclic} either $N_G(U)=N_{G_i^g}(U)$ for some $i$ and $g$, and so $K\cap N=K\cap N_{G_i^g}(U)$ is finitely generated,  or $N_G(U)$ is metacyclic and therefore so is $K\cap N$. 
 
 If $U$ stabilizes an edge then $U\leq C^g$ for some $g\in G$ and so we
 may assume that $U\leq C$. Then by Proposition \ref{normalizers}  $N$
 satisfies the hypothesis of either Theorem \ref{R1} or Theorem \ref{R2} and so by one of these theorems $N\cap K$ is finitely generated.
 \end{proof}

\end{document}